\documentclass[reqno,11pt]{amsart}
\usepackage{cite}

\usepackage{newtxtext}
\usepackage{mathptmx}
\DeclareMathAlphabet{\mathcal}{OMS}{cmsy}{m}{n}
\DeclareSymbolFont{largesymbols}{OMX}{cmex}{m}{n}

\usepackage[unicode, colorlinks]{hyperref}
\hypersetup{
	linkcolor = blue,
	citecolor = red,
	urlcolor = teal,
	bookmarksnumbered = true,
}

\usepackage{amsmath, amsthm, amssymb, bm, graphicx, mathrsfs}
\usepackage{cases}

\newtheorem{theorem}{Theorem}[section]
 
\newtheorem{proposition}[theorem]{Proposition}
\newtheorem{lemma}[theorem]{Lemma}

\newtheorem{corollary}[theorem]{Corollary}
\newtheorem{remark}[theorem]{Remark}

\newcommand{\e}{{\rm e}}
\newcommand{\ii}{{\rm i}}

\begin{document}
\title{Physical Space Proof of Bilinear Estimates and Applications to Nonlinear Dispersive Equations (II)}
\author{Xinfeng Hu}
\address{School of Mathematical Sciences, Fudan University, Shanghai 200433, China}
\email{xfhu25@m.fudan.edu.cn}

\author{Li Tu}
\address{School of Mathematical Sciences, Fudan University, Shanghai 200433, China}
\email{ltu23@m.fudan.edu.cn}

\author{Yi Zhou}
\address{School of Mathematical Sciences, Fudan University, Shanghai 200433, China}
\email{yizhou@fudan.edu.cn}


\date{\today}

\begin{abstract}
    The work by Kenig-Ponce-Vega \cite{KPV1} initiated the use of Bourgain spaces to study the low-regularity well-posedness of semilinear dispersive equations.
    Since then, the Bourgain space method has become the dominant, and almost the only method to deal with this problem.
    The goal of this series of papers is to propose an alternative approach for this problem that does not rely on Bourgain spaces.
    Our method is based on a bilinear estimate, which is proved in a physical space approach by a new div-curl type lemma introduced by the third author.
    Combining these ingredients with a Strichartz estimate of mixed spatial integrability, we will illustrate our method in the present paper by reproducing best known local well-posedness results for the $2d$ and $3d$ Zakharov system from Bejenaru-Herr-Holmer-Tataru \cite{2009On} and Bejenaru-Herr \cite{2010Convolutions}.
\end{abstract}

\maketitle
\setcounter{tocdepth}{1}
\tableofcontents

\section{Introduction}
Since the pioneering work of Kenig-Ponce-Vega \cite{KPV1}, the Bourgain space method has become the main method for establishing the low-regularity well-posedness of semilinear dispersive equations.
While other methods have since been developed, the Bourgain space method remains so profoundly influential that it is almost regarded as the only viable approach for this problem.
In this series of papers we aim to demonstrate that our recently developed bilinear estimate method, based on a new type of div-curl lemma due to the third author, can be an alternative. 
In fact, we believe that many classical results admit new proofs within our framework.

In the first part \cite{tu2025physicalspaceproofbilinear}, we consider the modified Korteweg-de Vries and modified Benjamin-Ono equations. 
In the following, we study the initial value problem for the Zakharov system: 
\begin{equation}\label{Zakharov equ}
    \begin{cases}
        \mathrm{i}\partial_{t}E+\Delta E=nE,\\
        \partial_{t}^2n-\Delta n=\Delta \vert E\vert^{2},\\
        t=0:(E,n,\partial_{t}n)=(E_{0},n_{0},n_{1}),
    \end{cases}
\end{equation}
where $E\colon [0, T] \times \mathbb{R}^d \rightarrow \mathbb{C}, n\colon [0, T] \times \mathbb{R}^d \rightarrow \mathbb{R}$, and the initial data $(E_0, n_0, n_1)$ belongs to Sobolev space $H^{s}(\mathbb{R}^{d})\times H^{l}(\mathbb{R}^{d})\times H^{l-1}(\mathbb{R}^{d})$.

The Zakharov system was introduced as a model for the propagation of Langmuir waves in a plasma \cite{1972Collapse}.
It satisfies conservation of mass
\begin{equation*}
    \mathcal{M}(E(t)) := \Vert E(t)\Vert_{L_x^2}^2 = \Vert E_0\Vert_{L_x^2}^2 = \mathcal{M}(E(0))
\end{equation*}
and energy
\begin{align*}
    \mathcal{E}(E(t), n(t), \partial_tn(t)) &:= \Vert \nabla_xE(t)\Vert_{L_x^2}^2 + \frac{1}{2}\left(\Vert n(t)\Vert_{L_x^2}^2 + \Vert \nabla_x^{-1}\partial_tn(t)\Vert_{L_x^2}^2\right) + \int_{\mathbb{R}^d}n|E|^2 \,{\rm d}x \\
    &= \mathcal{E}(E_0, n_0, n_1).
\end{align*}

The low-regularity problem for \eqref{Zakharov equ} can be traced back to the work of Bourgain-Colliander\cite{1996On}, in which global well-posedness in the energy space was established under a smallness condition.
In one spatial dimension ($d=1$), Colliander-Holmer-Tzirakis\cite{2006Low} proved that \eqref{Zakharov equ} is globally well-posed in $ L^{2}(\mathbb{R})\times H^{-1/2}(\mathbb{R})\times H^{-3/2}(\mathbb{R})$.
The local well-posedness for $2d$ Zakharov system in $H^{1/2}(\mathbb{R}^2)\times L^{2}(\mathbb{R}^2)\times H^{-1}(\mathbb{R}^2)$ was obtained by Ginibre-Tsutsumi-Velo\cite{GINIBRE1997384}, and this result was later improved to the optimal regularity $L^2(\mathbb{R}^2) \times H^{-1/2}(\mathbb{R}^2) \times H^{-3/2}(\mathbb{R}^2)$ by Bejenaru-Herr-Holmer-Tataru \cite{2009On}.
For the three-dimensional case, Bejenaru-Herr\cite{2010Convolutions} proved \eqref{Zakharov equ} is locally well-posed in $H^s(\mathbb{R}^3)\times H^{l}(\mathbb{R}^3)\times H^{l - 1}(\mathbb{R}^3)$ for $l > -1/2, l \leq s \leq l + 1, 2s > l + 1/2$.
Very recently, the sharp range of $(s, l)$ for well-posedness of \eqref{Zakharov equ} is determined for spatial dimensions $d \geq 4$ by Candy-Herr-Nakanishi \cite{TSK}.
We also refer the reader to other related works \cite{FPZ,GLNW,GN,TSK2}.

In this paper, we prove the following local well-posedness results for the $2d$ and $3d$ Zakharov system without using Bourgain spaces. 
Our results precisely match the sharpest known regularity thresholds for the initial data (see \cite{2009On,2010Convolutions}). 

\begin{theorem}\label{thm}
    Consider the Zakharov system \eqref{Zakharov equ} with initial data $(E_{0},n_{0},n_{1})\in H^{s}(\mathbb{R}^{d})\times H^{l}(\mathbb{R}^{d})\times H^{l-1}(\mathbb{R}^{d})$
    (the spaces $S_{1}(T)$ and $N_{1}(T)$ will be defined in the following section).
    \begin{enumerate}
        \item When $d=2, s=0, l=-1/2$, there exists $T=T\left(\Vert E_{0}\Vert_{L^2_x},\Vert n_{0}\Vert_{H^{-1/2}_x},\Vert n_{1}\Vert_{H^{-3/2}_x}\right) > 0$, such that \eqref{Zakharov equ} has a unique solution $(E, n, \partial_tn)$ satisfying
    \begin{gather}
        \label{a}E\in C([0,T];L^{2}(\mathbb{R}^{2})) \cap S_1(T), \\
        \label{b}n\in C([0,T];H^{-\frac{1}{2}}(\mathbb{R}^{2})), \quad \partial_tn \in C([0,T];H^{-\frac{3}{2}}(\mathbb{R}^{2})), \\ 
        \label{a1} nE \in N_1(T).       
    \end{gather}
    Moreover, for any neighborhood $U$ of initial data in $L^2(\mathbb{R}^2) \times H^{-1/2}(\mathbb{R}^2) \times H^{-3/2}(\mathbb{R}^2)$ the data-to-solution map $(E_0, n_0, n_1) \mapsto (E, n, \partial_tn)$ from $U$ into the class defined by \eqref{a}-\eqref{a1} is Lipschitz continuous.
    \item When $d=3, s > 0, l = s-1/2$, there exists $T=T\left(\Vert E_{0}\Vert_{H^{s}_x},\Vert n_{0}\Vert_{H^{l}_x},\Vert n_{1}\Vert_{H^{l-1}_x}\right) > 0$, such that \eqref{Zakharov equ} has a unique solution $(E,n)$ satisfying
    \begin{gather}
        \label{c}E\in C([0,T];H^{s}(\mathbb{R}^{3}))\cap S_1(T), \\
        \label{d}n\in C([0,T];H^{l}(\mathbb{R}^{3})), \quad \partial_tn \in C([0,T];H^{l - 1}(\mathbb{R}^{3})), \\
        \label{c1} nE \in N_1(T).      
    \end{gather}
    Moreover, for any neighborhood $U$ of initial data in $H^s(\mathbb{R}^3) \times H^{l}(\mathbb{R}^3) \times H^{l - 1}(\mathbb{R}^3)$ the data-to-solution map $(E_0, n_0, n_1) \mapsto (E, n, \partial_tn)$ from $U$ into the class defined by \eqref{c}-\eqref{c1} is Lipschitz continuous.
    \end{enumerate}
\end{theorem}

The basic idea to prove Theorem \ref{thm} is straightforward, namely, to establish contraction of the corresponding Picard iteration sequence.
Instead of using Bourgain spaces to control the nonlinearity, we employ the bilinear estimate method developed in \cite{tu2025physicalspaceproofbilinear} to control the nonlinear terms in the iteration scheme.

To be more specific, we recall the classical Strichartz estimate for Schr\"odinger equation (see \cite{2006Nonlinear}; here we take $d = 2$ as an example) and energy estimate for wave equation:

\begin{theorem}
    Suppose $E\colon [0, T] \times \mathbb{R}^2 \rightarrow \mathbb{C}$ is solution to $\ii \partial_tE + \Delta E = F$ with initial data $E(0) = E_0$. 
    Then 
    \begin{equation}\label{AB}
        \Vert E\Vert_{L_{t, x}^4} + \Vert E\Vert_{L_t^{\infty}L_x^2} \lesssim \Vert E_0\Vert_{L_x^2} + \Vert F\Vert_{L_{t, x}^{\frac{4}{3}}}.
    \end{equation}
\end{theorem}

\begin{theorem}
    Suppose $n\colon [0, T] \times \mathbb{R}^2 \rightarrow \mathbb{R}$ is a solution to $\partial_t^2n - \Delta n = G$ with initial data $n(0) = n_0, \partial_tn(0) = n_1$.
    Then 
    \begin{equation}\label{CD}
        \Vert n\Vert_{L_t^{\infty}L_x^2} + \Vert \partial_tn\Vert_{L_t^{\infty}H_x^{-1}} \lesssim (1 + T)\left(\Vert n_0\Vert_{L_x^2} + \Vert n_1\Vert_{H_x^{-1}} + \Vert G\Vert_{L_t^1H_x^{-1}}\right).
    \end{equation}
\end{theorem}

By Bony's paraproduct decomposition, the interactions in (quadratic) nonlinearity can be classified into three cases: high-low, low-high and high-high interactions.
The high-high interaction term can be controlled directly using Strichartz estimates.
For low-high and high-low interactions, we use an additional bilinear estimate which yields improved regularity.
As demonstrated in \cite{tu2025physicalspaceproofbilinear}, in some sense this results in a favorable average redistribution of derivatives in the nonlinearity.

Back to the Zakharov system \eqref{Zakharov equ}, if we apply the bilinear estimate method directly, by \eqref{AB} and \eqref{CD} we have
\begin{equation*}
    nE \in L_{t, x}^{\frac{4}{3}}, \qquad \Delta |E|^2 \in L_t^1H_x^{-1},
\end{equation*}
which also implies a local-in-time solution $(E, n, \partial_tn)$ in $(L_t^{\infty}H^{1/2}_x \cap L_{t, x}^4) \times L_t^{\infty}L^2_x \times L_t^{\infty}H^{-1}_x$.
However, this result still has $1/2$-derivative gap from the optimal regularity in \cite{2009On}. 
To overcome this, we further decompose the Schr\"odinger component of the solution according to different directions in phase space, where a better Strichartz estimate is given in the \emph{main direction} (see Section \ref{Section direction} and Theorem \ref{Strichartz e}). 

The bilinear estimates used here will be proved in a \emph{physical space approach}, relying on a new type of div-curl lemma first introduced by the third author \cite{Zhou1} and further developed in \cite{WZ1,WZ2,WZ3}.
This bilinear estimate method has shown to be a powerful tool in well-posedness theory of dispersive equations; see \cite{LSZ,LZ1,tu2025physicalspaceproofbilinear}.

Finally, we point out that the Picard iteration sequence we will use is constructed beyond the following reduced system, which is more convenient to work with.
Suppose that $(u, n, \partial_tn)$ is a sufficiently smooth solution to \eqref{Zakharov equ}. 
Define $\Lambda:=(-\Delta_x)^{1/2}$ and $v:=n+\mathrm{i}\Lambda^{-1}\partial_{t}n$.
Direct computation shows that $(E, v)$ satisfies
\begin{equation}\label{reduced equations}
    \begin{cases}
       \mathrm{i}\partial_{t}E+\Delta E=(\Re v) E,\\
        \mathrm{i}\partial_{t}v-\Lambda v=\Lambda \vert E\vert^{2}, \\
        t=0:(E,v)=(E_0, v_0) \in H^s(\mathbb{R}^d) \times H^l(\mathbb{R}^d),
    \end{cases}
\end{equation}
where $v_0 = n_{0}+\mathrm{i}\Lambda^{-1}n_{1}$. 
The Picard iteration scheme for this system is then set up as follows:
\begin{equation}\label{iteration}
    \begin{cases}
    \mathrm{i}\partial_{t}E^{(0)}+\Delta E^{(0)}=0,\\
    \mathrm{i}\partial_{t}v^{(0)}-\Lambda v^{(0)}=0,\\
    \mathrm{i} \partial_{t}E^{(k)}+\Delta E^{(k)}=(\Re v^{(k-1)})E^{(k-1)} ,k\geq1,\\
    \mathrm{i}\partial_{t}v^{(k)}-\Lambda v^{(k)}=\Lambda \vert E^{(k-1)}\vert^{2},k\geq1,\\    
    t=0:(E^{(k)},v^{(k)})=(E_{0}, v_0).
\end{cases}
\end{equation}
In this iteration process, we also have the mass conservation
\begin{equation}\label{mass conservation}
    \Vert E^{(k)}(t)\Vert_{L_{x}^{2}}=\Vert E_{0}\Vert_{L_{x}^{2}}, \qquad \forall t.
\end{equation}

Now Theorem \ref{thm} can be reduced to the proposition below: 

\begin{proposition}\label{prop}
    Let $\{E^{(k)}\}_{k=0}^{\infty}$ and $\{v^{(k)}\}_{k=0}^{\infty}$ be the iteration sequences defined by \eqref{iteration}.
    Under the assumptions
    \begin{enumerate}
        \item $d=2, s=0, l=-1/2$, or 
        \item $d=3, s>0, l=s- 1/2$,
    \end{enumerate}
    there exists $T=T\left(\Vert E_{0}\Vert_{H^{s}_x},\Vert v_{0}\Vert_{H^{l}_x}\right) > 0$ and a constant $C = C\left(\Vert E_{0}\Vert_{H^{s}_x},\Vert v_{0}\Vert_{H^{l}_x}\right) > 0$, such that for all $k = 0, 1, 2, \cdots$ we have 
\begin{gather}
        \label{bound}\Vert E^{(k)}\Vert_{L_t^{\infty}H_x^s} + \Vert E^{(k)}\Vert_{S_1(T)} \leq C, \\ 
        \label{17}\Vert \Re v^{(k)}\Vert_{S_2(T)} \leq C, \\ 
        \label{18}\Vert (\Re v^{(k)})E^{(k)}\Vert_{N_1(T)} \leq C, \\ 
        \Vert E^{(k + 3)}-E^{(k + 2)}\Vert_{L_t^{\infty}H_x^s}+\Vert E^{(k + 3)}-E^{(k + 2)}\Vert_{S_1(T)} \leq \frac{1}{2^k}, \label{E_Hs}\\ 
        \Vert \Re v^{(k + 3)}-\Re v^{(k + 2)}\Vert_{S_2(T)} \leq \frac{1}{2^{k}}\label{v_S2},\\
        \Vert (\Re v^{(k + 3)})E^{(k + 3)} - (\Re v^{(k + 2)})E^{(k + 2)}\Vert_{N_1(T)} \leq \frac{1}{2^k}.
\end{gather}
\end{proposition}

\section{Preliminaries}\label{Notations}

We write $X\lesssim Y$ to indicate $X\leq CY$, and $X\gtrsim Y$ indicate $X\geq CY$ for some harmless positive constant $C$. The symbol $X\sim Y$ implies that $X\lesssim Y$ and $X\gtrsim Y$.
The Greek letters $\lambda,\mu$ and $\sigma$ are used to denote dyadic numbers $\{2^k\colon k \in \mathbb{N}\}$, and $\omega$ the elements in $\mathbb{S}^{d - 1}$ in this paper.
The convolution of two functions $u$ and $v$ are defined as 
\begin{equation*}
    (u \ast v)(x) = \int u(x - y)v(y) dy.
\end{equation*}

For $s \in \mathbb{R}$ and $1 \leq p, q \leq \infty$, we use $L^{p}$ to denote the usual Lebesgue spaces, and $W^{s,p}$ with $H^{s}=W^{s,2}$ for the usual Sobolev spaces. 
The spacetime norm $\Vert \cdot\Vert _{L_{t}^{q}L_{x}^{p}}$ is defined by
\begin{equation*}
    \Vert u\Vert_{L_{t}^{q}L_{x}^{p}}=\Vert u\Vert_{L^q(0,T;L^{p}(\mathbb{R}^d))}:=\left(\int_{0}^{T}\left(\int_{\mathbb{R}^{d}}\vert u(t, x)\vert ^{p}dx\right)^{\frac{q}{p}}dt\right)^{\frac{1}{q}}
\end{equation*}
with usual modifications when $p = \infty$ or $q = \infty$.
When $p=q$, we will abbreviate $L_{t}^{p}L_{x}^{p}$ as $L_{t,x}^{p}$.

Given a function $u=u(t,x)$, the spatial Fourier transform of $u$ is defined by
\begin{equation*}
    \widehat{u}(\xi)=\mathcal{F}_{x}u(\xi):=\int_{\mathbb{R}^{d}}\e^{-\ii x\cdot\xi}u(x)dx,
\end{equation*}
and the Fourier inversion with respect to variable $\xi$:
\begin{equation*}
    \mathcal{F}^{-1}_{\xi}u(x):=\frac{1}{(2\pi)^{d}}\int_{\mathbb{R}^{d}}\e^{\ii x\cdot\xi}\widehat{u}(\xi)d\xi.
\end{equation*}

Let $\psi$ be a smooth non-neagtive radial function on $\mathbb{R}$ such that ${\rm supp}\ \psi\subset \{8/9\leq \vert \xi \vert\leq 9/8\}$. 
Set $\varphi(\xi)=\psi(|\xi|/2)-\psi(|\xi|)$ for $\xi \in \mathbb{R}^d$. For dyadic number $\lambda$ , we define
\begin{equation*}
    \widehat{P_{\lambda}u}(\xi) := \varphi(\lambda^{-1}\xi)\widehat{u}(\xi), \quad P_{\leq \lambda} := \sum_{\mu \leq \lambda}P_{\mu},
\end{equation*}
then we have standard inhomogeneous Littlewood-Paley decomposition
\begin{equation}\label{littewood-paley}
    u=P_{\leq 1}u + \sum_{\lambda > 1}P_{\lambda}u,
\end{equation}
With a slight abuse of notations we also write $u = \sum_{\lambda \geq 1}P_{\lambda}u$.
For the product of two functions $u$ and $v$, we write Bony's paraproduct decomposition as
\begin{equation}\label{bony}
    P_{\sigma}(uv) = P_{\sigma}uP_{\ll \sigma}v + P_{\ll \sigma}uP_{\sigma}v +\sum_{1\leq \sigma\lesssim\lambda\sim\mu} P_{\lambda}u P_{\mu}v.
\end{equation}
For $s\in \mathbb{R}$, $1\leq p\leq q \leq \infty$, the following holds, which is also known as Bernstein's inequality:
\begin{gather}\label{bernstein}
    \Vert \Lambda^sP_{\lambda}u\Vert_{L^{p}_{x}}\lesssim\lambda^{s+d\left(\frac{1}{p}-\frac{1}{q}\right)}\Vert P_{\lambda}u\Vert_{L^{q}_{x}}.
\end{gather}
The $H^s_x$ norm can be expressed in dyadic settings, i.e., 
\begin{equation*}
    \Vert u \Vert_{H^s_x} \sim \left(\sum_{\lambda \geq 1}\left(\lambda^s\Vert P_{\lambda}u\Vert_{L^2_x}\right)^2\right)^{\frac{1}{2}}.
\end{equation*}

\subsection{Decomposition of directions in phase space}\label{Section direction}
We decompose $P_{\lambda}u$ as
\begin{equation*}
     P_{\lambda}u=\sum\limits_{i=1}^{N} P_{\lambda,\omega_{i}}u,
\end{equation*}
where $N$ is a finite number depending only on dimension $d$, $\widehat{P_{\lambda,\omega_{i}}u}(\xi)=Q_{\omega_i}(\xi/\vert \xi\vert)\widehat{P_{\lambda}u}(\xi)$, $\{Q_{\omega_i}\}_{i=1}^{N}$ is the partition of unity of $\mathbb{S}^{d - 1}$ and 
in the support of $Q_{\omega_i}$, there holds
\begin{equation}\label{30}
   |\xi \cdot \omega_i| \geq 100|\xi'|.
\end{equation}
Here $\xi \in \mathbb{R}^d, \omega_i \in \mathbb{S}^{d - 1}$, $\xi'$ is the orthogonal projection of $\xi$ onto $H_{\omega_i}$, and $H_{\omega_i}$ is the orthogonal complement of $\omega_i$ in $\mathbb{R}^d$.

This decomposition is crucial in the proof of bilinear estimates for the high-low interaction between the Schr\"odinger and wave components of the solution (see Theorem \ref{theorem H-L H-H}).
In terms of the definition above, for every $P_{\lambda, \omega_i}$ we call $\omega_i$ the main direction.

\begin{remark}
    If $\omega_i = e_i = (0, \cdots, 0, 1, 0, \cdots, 0)$, where the $i$th component of $e_i$ is $1$ and others are $0$, it follows from \eqref{30} that 
\begin{equation*}
    {\rm supp}\ \widehat{P_{\lambda, e_i}u} \subset \left\{\xi = (\xi^{(1)}, \cdots, \xi^{(d)}) \in \mathbb{R}^d\colon |\xi| \sim \lambda, |\xi^{(i)}| \geq 100|\xi'|\right\}.
\end{equation*}
\end{remark}

\subsection{Iteration spaces}\label{Function Space}

In the following we fix a partition of unity $\{Q_{\omega_i}\}_{i = 1}^N$ of $\mathbb{S}^{d - 1}$.
For $1 < p, q < \infty$ and every $i = 1, \cdots, N$ we define 
\begin{equation}\label{Lpq}
    \Vert u\Vert_{L_{\omega_i}^qL_{\omega_{i\bot}}^p} := \left(\int_{\mathbb{R}}\left(\int_{H_{\omega_i}}|u(a\omega_i + x')|^p dx'\right)^{\frac{q}{p}} da\right)^{\frac{1}{q}}
\end{equation}
with usual modifications when $p = \infty$ or $q = \infty$. 

For $d=2$, we set
\begin{equation*}
    \Vert u\Vert_{S_1(T)} := \left(\sum_{i = 1}^N\sum_{\lambda\geq 1}\Vert P_{\lambda}u\Vert_{L_{t}^{4}L_{\omega_i}^{2}L_{\omega_{i\bot}}^{\infty}}^{2}\right)^{\frac{1}{2}}.
\end{equation*}
Then $S_1(T)$ is defined as the completion of all test functions under this norm.
We denote by 
\begin{equation*}
    \Vert u\Vert_{N_1(T)} := \left(\sum_{i = 1}^N\sum_{\lambda\geq 1}\Vert P_{\lambda}u\Vert_{L_{t}^{\frac{4}{3}}L_{\omega_i}^{2}L_{\omega_{i\bot}}^{1}}^{2}\right)^{\frac{1}{2}}.
\end{equation*}
For simplicity, we also set 
\begin{gather*}
    S_2(T) := L_{t}^{\infty}H^{-\frac{1}{2}}_x, \quad N_2(T) := L_t^1H_x^{-\frac{1}{2}}, \quad X(T) := L^{\infty}_tL_x^2 \cap S_1(T).
\end{gather*}

For $d = 3$, we define ($s > 0$)
\begin{gather*}
    \Vert u\Vert_{S_1(T)} := \left(\sum_{i = 1}^N\sum_{\lambda\geq 1}\Vert P_{\lambda}u\Vert_{L_{t}^{2}L_{\omega_i}^{2}L_{\omega_{i\bot}}^{\infty}}^{2}\right)^{\frac{1}{2}}, \\ 
    \Vert u\Vert_{N_1(T)} := \left(\sum_{i = 1}^N\sum_{\lambda\geq 1}\lambda^{4s}\Vert P_{\lambda}u\Vert_{L_{t}^{2}L_{\omega_i}^{2}L_{\omega_{i\bot}}^1}^{2}\right)^{\frac{1}{2}}, \\
    S_2(T) := L_t^{\infty}H_x^{s - \frac{1}{2}}, \\ 
    N_2(T) := L_t^1H_x^{s - \frac{1}{2}}, \\
    X(T) := L_t^{\infty}H_x^s \cap S_1(T).
\end{gather*}

\begin{remark}
    If $\omega_i = e_i$, by \eqref{Lpq} we have 
\begin{equation*}
    \Vert u\Vert_{L_{e_i}^qL_{e_i\bot}^p} = \left(\int_{\mathbb{R}}\left(\int_{\mathbb{R}^{d - 1}}|u|^p dy\right)^{\frac{q}{p}} dx_i\right)^{\frac{1}{q}} =: \Vert u\Vert_{L_{x_i}^qL_y^p} \qquad (x_i \in \mathbb{R}, y \in \mathbb{R}^{d - 1}).
\end{equation*}
\end{remark}

\begin{remark}
    It is worth mentioning that the $L_{\omega_i}^qL_{\omega_{i\bot}}^p$ norm \eqref{Lpq} is analogous in form to the lateral spacetime norm $\Vert \cdot \Vert_{L_{\mathrm{e}}^{p, q}}\ (\mathrm{e} \in \mathbb{S}^{d - 1})$ used in the study of Schr\"odinger map \cite{IK1,BIKT}.
    However, there is a fundamental difference since our norms are defined in the way where $\omega_i$ serves as the main direction.
\end{remark}

\section{Refined Strichartz Estimates}\label{Section Strichartz}
The Schr\"odinger semigroup $\{\e^{\ii t\Delta}\}_{t > 0}$ is defined as 
\begin{equation*}
    \widehat{\e^{\ii t\Delta}E_0}(t, \xi) := \e^{-\ii t|\xi|^2}\widehat{E_0}(\xi).
\end{equation*}
It obeys conservation of mass
\begin{equation}\label{mass}
    \Vert \e^{\ii t\Delta}E_0\Vert_{L_x^2} = \Vert E_0\Vert_{L_x^2}.
\end{equation}
% To prove the desired estimate for $\{\e^{\ii t\Delta}\}_{t > 0}$, 
We first present the following lemma concerning its decay property at infinity.

\begin{lemma}[$L^1-L^{\infty}$ estimate]\label{l1-linfty lemma}
Let $E\colon [0, T] \times \mathbb{R}^d \rightarrow \mathbb{C}$ satisfy
           \begin{equation}\label{strichartz_equ1}
       \begin{cases}
           {\rm i}\partial_{t}E+\Delta E= 0,\\
           t = 0\colon E = E_0. 
       \end{cases}
   \end{equation}
    There holds 
    \begin{equation*}
        \Vert P_{\lambda}E(x_1, \cdot)\Vert_{L_y^{\infty}(\mathbb{R}^{d - 1})} \lesssim \int_{\mathbb{R}}(\lambda^{-1} + \sqrt{t} + |x_1-x_{1}'|)^{-d}\Vert P_{\lambda}E_0(x_1', \cdot)\Vert_{L_y^1(\mathbb{R}^{d - 1})}dx_{1}',
    \end{equation*}
    where $x = (x_1, y) \in \mathbb{R} \times \mathbb{R}^{d - 1}$, $\lambda \geq 1$ dyadic.
    \begin{proof}
        Since ${\rm supp}\ \widehat{P_{\lambda}E} \subset \{|\xi| \sim \lambda\}$, it follows that 
        \begin{equation*}
            \displaystyle\mathcal{F}_{x}\left({\e^{\ii t\Delta}P_{\lambda}E_0}\right)(t, \xi) = \e^{-\ii t|\xi|^2}\widehat{P_{\lambda}E_0}(\xi) = \e^{-\ii t|\xi|^2}\chi\left(\frac{\xi}{\lambda}\right)\widehat{P_{\lambda}E_0}(\xi),
        \end{equation*}
        where $\chi$ is a smooth cut-off satisfying $\mathrm{supp}\ \chi\supset \mathrm{supp}\ \widehat{P_{1}E_{0}}$, and we have 
        \begin{equation*}
            \e^{\ii t\Delta}P_{\lambda}E_0 = \varphi_{\lambda}(t, \cdot) \ast P_{\lambda}E_0, \qquad \widehat{\varphi_{\lambda}}(t, \xi) = \e^{-\ii t|\xi|^2}\chi\left(\frac{\xi}{\lambda}\right).
        \end{equation*}
        By Minkowski's inequality and a scaling argument, it suffices to prove
        \begin{equation*}
            \Vert \varphi_1(t, x_1, \cdot)\Vert_{L_y^{\infty}(\mathbb{R}^{d - 1})} \lesssim (1 + \sqrt{t} + |x_1|)^{-d}.
        \end{equation*}
        
        Firstly, by \eqref{mass} and Bernstein's inequality we have (note that $\lambda = 1$)
        \begin{equation}\label{38}
            \begin{aligned}
            \Vert \varphi_1(t,x_1,\cdot)\Vert_{L^{\infty}_{y}(\mathbb{R}^{d - 1})}\lesssim \Vert \varphi_1(t,x_1,\cdot)\Vert_{L^2_{y}(\mathbb{R}^{d - 1})} = \Vert \mathcal{F}_{\xi}^{-1}\chi\Vert_{L_y^2(\mathbb{R}^{d - 1})} \lesssim 1. 
            \end{aligned}
        \end{equation}
        
        Secondly, direct computations show that (see \cite{stein1993harmonic}) 
        \begin{equation*}
            \varphi_1(t,x)=\left(\frac{\mathrm{1}}{(\pi \mathrm{i} t)^{\frac{d}{2}}}\e^{\frac{\mathrm{i}\vert \cdot\vert^{2}}{t}}* \mathcal{F}^{-1}_{\xi}{\chi}\right)(x).
        \end{equation*}
        It follows from Minkowski's inequality that 
        \begin{equation}\label{l1-linfty}
            \Vert \varphi_1(t,x_1, \cdot)\Vert_{L^{\infty}_{y}(\mathbb{R}^{d - 1})} \lesssim t^{-\frac{d}{2}}\Vert \mathcal{F}^{-1}_{\xi}\chi\Vert_{L^{1}(\mathbb{R}^{d})}\lesssim t^{-\frac{d}{2}}.
        \end{equation}
        
        Finally, we set 
        \begin{gather*}
            \phi(\xi)=x\cdot\xi-t\vert \xi\vert^{2}, \\ 
            D=\{\xi\in\mathbb{R}^{d}:\vert x_1-2t\xi_{1}\vert\geq \vert x_1\vert/2 \},
        \end{gather*}
        and decompose $\varphi_{1}$ into two parts:
        \begin{align*}
        \varphi_1(t,x)=\int_{\mathbb{R}^{d}}\e^{\ii\phi(\xi)}\chi(\xi)d\xi =\int_{D}\e^{\ii\phi(\xi)}\chi(\xi)d\xi+\int_{\mathbb{R}^{d}\setminus D}\e^{\ii\phi(\xi)}\chi(\xi)d\xi =:\varphi_{1,1}(t,x)+\varphi_{1,2}(t,x).
        \end{align*}
        For $\varphi_{1,1}(t,x)$, note that 
        \begin{equation*}
            \e^{\ii\phi} = \frac{1}{\ii(x_1 - 2t\xi_1)}\frac{\partial\e^{\ii\phi}}{\partial\xi_1}.
        \end{equation*}
        Then 
        \begin{equation*}
            \varphi_{1, 1}(t, x) = \int_D\frac{1}{\ii(x_1 - 2t\xi_1)}\frac{\partial\e^{\ii\phi}}{\partial\xi_1}\chi(\xi) d\xi.
        \end{equation*}
        Integration by parts $d$ times with respect to $\xi_{1}$ shows that
        \begin{equation*}
        \Vert \varphi_{1,1}(t,x_1,\cdot)\Vert_{L_{y}^{\infty}}\lesssim \left\Vert|x_1 - 2t\xi_{1}|^{-d}\right\Vert_{L^{\infty}_{\xi_{1}}} \lesssim \vert x_1\vert^{-d}.
    \end{equation*}
    For $\varphi_{1,2}$, we further rescale it as follows: Set
    \begin{equation*}
        \Phi_{1, 2}(s, \widetilde{x_1}, \widetilde{y}) = \varphi_{1, 2}(ts, \sqrt{t}x_1, \sqrt{t}y).
    \end{equation*}
    Then $\Phi_{1, 2}$ satisfies 
    \begin{equation*}
        \begin{cases}
        \ii\partial_s\Phi + \widetilde{\Delta}\Phi = 0, \\ 
        s = 0\colon \Phi = \mathscr{F}_{\xi}^{-1}(\chi(\sqrt{t} \cdot)),
        \end{cases}
    \end{equation*}
    where $\widetilde{\Delta} = \sum_{i = 1}^d\partial_{\widetilde{x_i}}^2$.
    By \eqref{l1-linfty} we have 
    \begin{equation}\label{52}
        \Vert \Phi(s, \widetilde{x_1}, \cdot)\Vert_{L_{\widetilde{y}}^{\infty}} \lesssim 1^{-\frac{d}{2}} \cdot (\sqrt{t})^{-d} = t^{-\frac{d}{2}}.
    \end{equation}
    Note that in the support of $\Phi_{1, 2}$,
    \begin{equation*}
        |\widetilde{x_1} - 2\widetilde{\xi_1}| \leq \frac{1}{2}|\widetilde{x_1}|.
    \end{equation*}
    It follows that 
    \begin{equation*}
        |\widetilde{x_1}| \sim |\widetilde{\xi_1}| \lesssim \widetilde{\lambda} = \sqrt{t}.
    \end{equation*}
    Now from \eqref{52}
    \begin{equation*}
        \Vert \Phi_{1, 2}(s, \widetilde{x_1}, \cdot)\Vert_{L_{\widetilde{y}}^{\infty}} \lesssim |\widetilde{x_1}|^{-d}.
    \end{equation*}
    Namely, 
    \begin{equation}\label{40}
        \Vert \varphi_{1, 2}(t, x_1, \cdot)\Vert_{L_y^{\infty}} \lesssim |x_1|^{-d}.
    \end{equation}
    Combining \eqref{38}, \eqref{l1-linfty} and \eqref{40}, we complete the proof.
    \end{proof}
\end{lemma}

\begin{theorem}[Refined Strichartz estimates]\label{Strichartz e}
    Let $E\colon [0, T] \times \mathbb{R}^d \rightarrow \mathbb{C}$ satisfy
           \begin{equation}\label{strichartz_equ}
       \begin{cases}
           {\rm i}\partial_{t}E+\Delta E= F,\\
           t = 0\colon E = E_0. 
       \end{cases}
   \end{equation}
   Set $x = (x_1, y) \in \mathbb{R} \times \mathbb{R}^{d - 1}$. Then we have 
   \begin{gather*}
    \Vert P_{\lambda}E\Vert_{L_t^4L_{x_1}^2L_y^{\infty}} \lesssim \Vert P_{\lambda}E_0\Vert_{L_x^2} + \Vert P_{\lambda}F\Vert_{L_t^{\frac{4}{3}}L_{x_1}^2L_y^1} \qquad (d = 2), \\ 
    \Vert P_{\lambda}E\Vert_{L_t^2L_{x_1}^2L_y^{\infty}} \lesssim (T\lambda^2)^{\varepsilon}\left(\Vert P_{\lambda}E_0\Vert_{L_x^2} + (T\lambda^2)^{\varepsilon}\Vert P_{\lambda}F\Vert_{L_t^2L_{x_1}^2L_y^1}\right) \qquad (d = 3),
   \end{gather*}
   where $\varepsilon > 0$ can be sufficiently small.
   \begin{proof}
    By Duhamel's principle, 
    \begin{equation*}
        P_{\lambda}E(t)=\e^{\ii t\Delta}P_{\lambda}E_0-\ii\int_{0}^{t}\e^{\ii(t-s)\Delta}P_{\lambda}F(s)ds, \qquad t \in [0, T].
    \end{equation*}
    We first consider the case $d = 2$. By duality and Christ-Kiselev lemma (see \cite{smithglobal}) it suffices to prove 
    \begin{equation}\label{F_strichartz}
        \left\Vert \int_0^T\e^{\ii(t-s)\Delta}P_{\lambda}F(s)ds\right\Vert_{L_{t}^{4}L_{x_1}^{2}L_{y}^{\infty}}\lesssim \Vert P_{\lambda}F\Vert_{L_{t}^{\frac{4}{3}}L_{x_1}^{2}L_{y}^{1}}.
    \end{equation}
    By Lemma \ref{l1-linfty lemma}, Minkowski's inequality, H\"older's inequality and Young's inequality of convolution type we have 
    \begin{equation*}
        \begin{aligned}
            \left\Vert\int_0^T\e^{\ii(t-s)\Delta}P_{\lambda}F(s)ds \right\Vert_{L_{x_1}^2L_y^{\infty}} &\lesssim \int_0^T\Vert \e^{\ii(t - s)\Delta}P_{\lambda}F(s)\Vert_{L_{x_1}^2L_y^{\infty}} ds \\
            &\lesssim \int_0^T\Vert (\lambda^{-1} + \sqrt{t - s} + |x_1|)^{-2}\Vert_{L_{x_1}^1}\Vert P_{\lambda}F(s)\Vert_{L_{x_1}^2L_y^1} ds \\ 
            &\lesssim \int_0^T (\lambda^{-1} + \sqrt{t - s})^{-1}\Vert P_{\lambda}F(s)\Vert_{L_{x_1}^2L_y^1} ds \\
            &\lesssim \int_{\mathbb{R}}|t - s|^{-\frac{1}{2}}\Vert P_{\lambda}F(s)\Vert_{L_{x_1}^2L_y^1} ds,
        \end{aligned}
    \end{equation*}
    and the desired conclusion follows from Hardy-Littlewood-Sobolev inequality.

    For the case $d = 3$, we first prove 
    \begin{equation*}
        \left\Vert \int_{0}^{t}\e^{\ii(t-s)\Delta}P_{\lambda}F(s)ds\right\Vert_{L_{t}^{2}L_{x_1}^{2}L_{y}^{\infty}}\lesssim (T\lambda^2)^{\varepsilon}\Vert P_{\lambda}F\Vert_{L_{t}^{2}L_{x_1}^{2}L_{y}^{1}}.
    \end{equation*}
    Indeed, similar to the case $d = 2$ we have 
    \begin{equation*}
        \begin{aligned}
            \left\Vert \int_{0}^{t}\e^{\ii(t-s)\Delta}P_{\lambda}F(s)ds\right\Vert_{L_{x_1}^{2}L_{y}^{\infty}}&\lesssim \int_0^t(\lambda^{-1} + \sqrt{t - s})^{-2}\Vert P_{\lambda}F(s)\Vert_{L_{x_1}^2L_y^1} ds \\ 
            &\lesssim \int_0^t(\lambda^{-2} + |t - s|)^{-1}\Vert P_{\lambda}F(s)\Vert_{L_{x_1}^2L_y^1} ds \\ 
            &\lesssim \int_0^T(\lambda^{-2} + |t - s|)^{-1}\Vert P_{\lambda}F(s)\Vert_{L_{x_1}^2L_y^1} ds.
        \end{aligned}
    \end{equation*}
    It follows from Young's inequality of convolution type that 
    \begin{equation*}
        \begin{aligned}
            \left\Vert \int_{0}^{t}\e^{\ii(t-s)\Delta}P_{\lambda}F(s)ds\right\Vert_{L_t^2L_{x_1}^{2}L_{y}^{\infty}}&\lesssim \ln\left(\frac{\lambda^{-2} + T}{\lambda^{-2}}\right)\Vert P_{\lambda}F\Vert_{L_t^2L_{x_1}^2L_y^1} \\ 
            &= \ln(1 + T\lambda^2)\Vert P_{\lambda}F\Vert_{L_t^2L_{x_1}^2L_y^1} \\ 
            &\lesssim (T\lambda^2)^{2\varepsilon}\Vert P_{\lambda}F\Vert_{L_t^2L_{x_1}^2L_y^1},
        \end{aligned}
    \end{equation*}
    where $\varepsilon > 0$ can be arbitrarily small. Moreover, the deductions above also imply the fact 
    \begin{equation*}
        \left\Vert \int_{0}^T\e^{\ii(t-s)\Delta}P_{\lambda}F(s)ds\right\Vert_{L_{t}^{2}L_{x_1}^{2}L_{y}^{\infty}}\lesssim (T\lambda^2)^{2\varepsilon}\Vert P_{\lambda}F\Vert_{L_{t}^{2}L_{x_1}^{2}L_{y}^{1}}.
    \end{equation*}
    Then by duality we have 
    \begin{equation*}
        \Vert \e^{\ii t\Delta}P_{\lambda}E_0\Vert_{L_t^2L_{x_1}^2L_y^{\infty}} \lesssim (T\lambda^2)^{\varepsilon}\Vert P_{\lambda}E_0\Vert_{L_x^2}.
    \end{equation*}
    This completes the proof.
   \end{proof}
\end{theorem}

Using the notations in Section \ref{Function Space}, by Theorem \ref{Strichartz e} we have 
\begin{gather}
    \label{stricahrtz d=2}\sum_{i = 1}^N\Vert P_{\lambda}E\Vert_{L_{t}^{4}L_{\omega_i}^{2}L_{\omega_{i\bot}}^{\infty}}\lesssim \Vert P_{\lambda}E_0\Vert_{L_{x}^{2}}+\sum_{i = 1}^N\Vert P_{\lambda}F\Vert_{L_{t}^{\frac{4}{3}}L_{\omega_i}^{2}L_{\omega_{i\bot}}^{1}} \ (d = 2), \\ 
    \label{stricahrtz d=3}\sum_{i = 1}^N\Vert P_{\lambda}E\Vert_{L_{t}^{2}L_{\omega_i}^{2}L_{\omega_{i\bot}}^{\infty}}\lesssim (T\lambda^2)^{\varepsilon}\left(\Vert P_{\lambda}E_0\Vert_{L^{2}_{x}}+(T\lambda^2)^{\varepsilon}\sum_{i = 1}^N\Vert P_{\lambda}F\Vert _{L_{t}^{2}L_{\omega_i}^{2}L_{\omega_{i\bot}}^{1}}\right) \ (d = 3, \varepsilon > 0).
\end{gather}

\section{Div-curl Lemma and Bilinear Estimates}\label{Section Div curl}

As mentioned earlier, we need extra bilinear estimates to handle the high-low and low-high interactions in the nonlinearity during the iteration. 
These bilinear estimates will be proved by the div-curl type lemma below.

\begin{lemma}[Div-curl lemma \cite{Zhou1,WZ1}]\label{div-curl}
    Suppose that 
    \begin{equation*}
        \begin{cases}
            \partial_{t}f^{11}+\partial_{x}f^{12}=G^{1},\\
            \partial_{t}f^{21}-\partial_{x}f^{22}=G^{2},
        \end{cases}
             \quad(t,x)\in[0,T]\times \mathbb{R}
    \end{equation*}
    and $f^{ij}\rightarrow 0 \ (i, j = 1, 2)$ as $x\rightarrow \infty$. Then
    \begin{equation*}
        \begin{aligned}
            \int_0^T\int_{\mathbb{R}}(f^{11}f^{22}+f^{12}f^{21})dxdt\lesssim&\left(\Vert f^{11}(0)\Vert_{L^{1}_{x}}+\Vert f^{11}\Vert_{L_{t}^{\infty}L_{x}^{1}}+\Vert G^{1}\Vert_{L_{t,x}^{1}}\right)\\
    \cdot&\left(\Vert f^{21}(0)\Vert_{L^{1}_{x}}+\Vert f^{21}\Vert_{L_{t}^{\infty}L_{x}^{1}}+\Vert G^{2}\Vert_{L_{t,x}^{1}}\right)
        \end{aligned}
    \end{equation*}
    provided that the right-hand side is bounded.
    \begin{proof}
        See \cite[Lemma 2.1]{WZ1}.
    \end{proof}
\end{lemma}

To apply the div-curl lemma above, we exploit the balance laws of Schr\"odinger and wave equations.

\begin{proposition}[Conservation of mass and momentum]\label{prop4.2}
    Suppose that $E$ and $n$ satisfy $\ii \partial_tE + \Delta E = 0$ and $\partial_t^2n - \Delta n = 0$, respectively.
    Then
    \begin{equation*}
        \begin{cases}
            \displaystyle\partial_t\int_{\mathbb{R}^{d - 1}}\frac{|E|^2}{2} dy - \partial_{x_1}\int_{\mathbb{R}^{d - 1}}\Im\left(E\overline{\partial_{x_1}E}\right) dy = 0, \\ 
            \displaystyle\partial_t\int_{\mathbb{R}^{d - 1}}\Im\left(E\overline{\partial_{x_1}E}\right) dy - \partial_{x_1}\int_{\mathbb{R}^{d - 1}}\left(2|\partial_{x_1}E|^2 - \partial_{x_1}^2\frac{|E|^2}{2}\right) dy = 0
        \end{cases}
    \end{equation*}
    and 
    \begin{equation*}
        \begin{cases}
            \displaystyle\partial_t\int_{\mathbb{R}^{d - 1}}e(n) dy - \partial_{x_1}\int_{\mathbb{R}^{d - 1}}(\partial_tn\partial_{x_1}n) dy = 0, \\ 
            \displaystyle\partial_t\int_{\mathbb{R}^{d - 1}}(\partial_tn\partial_{x_1}n) dy - \partial_{x_1}\int_{\mathbb{R}^{d - 1}}\frac{1}{2}\left(|\partial_tn|^2 + |\partial_{x_1}n|^2 - |\nabla_yn|^2\right) dy = 0,
        \end{cases}
    \end{equation*}
    where $x = (x_1, y) \in \mathbb{R} \times \mathbb{R}^{d - 1}, e(n) := \left(|\partial_tn|^2 + |\nabla_xn|^2\right)/2$.
    \begin{proof}
        Direct computations.
    \end{proof}
\end{proposition}

\begin{theorem}[Bilinear estimates]\label{theorem H-L H-H}
    Suppose that $E$ and $n$ satisfy 
    \begin{equation*}
        \begin{cases}
            \ii \partial_tE + \Delta E = 0, \\
            \partial_t^2n - \Delta n = 0, \\
            t = 0\colon (E, n, \partial_tn) = (E_0, n_0, n_1).
        \end{cases}
    \end{equation*}
    We have 
    \begin{enumerate}
        \item ($1 \ll \mu$)
            \begin{equation}\label{L-H interation}
            \begin{aligned}
                \Big\Vert \Vert P_{\lambda}n\Vert_{L_{y}^{2}}\Vert P_{\mu,e_{1}}E\Vert_{L_{y}^{2}}\Big\Vert_{L_{t, x_1}^{2}}\lesssim \mu^{-\frac{1}{2}}\Vert P_{\mu,e_{1}}E_0\Vert_{L_{x}^{2}}\left(\Vert P_{\lambda}n_0\Vert_{L_{x}^{2}}+\lambda^{-1}\Vert P_{\lambda}n_1\Vert_{L_{x}^{2}}\right),\\
                \Big\Vert \Vert P_{\lambda}\partial_{t}n\Vert_{L_{y}^{2}}\Vert P_{\mu,e_{1}}E\Vert_{L_{y}^{2}}\Big\Vert_{L_{t, x_1}^{2}}\lesssim \mu^{-\frac{1}{2}}\lambda\Vert P_{\mu,e_{1}}E_0\Vert_{L_{x}^{2}}\left(\Vert P_{\lambda}n_0\Vert_{L_{x}^{2}}+\lambda^{-1}\Vert P_{\lambda}n_1\Vert_{L_{x}^{2}}\right);
            \end{aligned}
    \end{equation}
    \item ($ \mu \ll \lambda$)
    \begin{equation}\label{H-L interaction}
    \begin{aligned}
        \Big\Vert \Vert P_{\lambda,e_{1}}n\Vert_{L_{y}^{2}}\Vert P_{\mu}E\Vert_{L_{y}^{2}}\Big\Vert_{L_{t,x_1}^{2}}\lesssim \lambda^{-\frac{1}{2}}\Vert P_{\mu}E_{0}\Vert_{L_{x}^{2}}\left(\Vert P_{\lambda,e_{1}}n_0\Vert_{L_{x}^{2}}+\lambda^{-1}\Vert P_{\lambda,e_{1}}n_1\Vert_{L_{x}^{2}}\right),\\
        \Big\Vert \Vert \partial_{t}P_{\lambda,e_{1}}n\Vert_{L_{y}^{2}}\Vert P_{\mu}E\Vert_{L_{y}^{2}}\Big\Vert_{L_{t,x_1}^{2}}\lesssim \lambda^{\frac{1}{2}}\Vert P_{\mu}E_0\Vert_{L_{x}^{2}}\left(\Vert P_{\lambda,e_{1}}n_0\Vert_{L_{x}^{2}}+\lambda^{-1}\Vert P_{\lambda,e_{1}}n_1\Vert_{L_{x}^{2}}\right),
    \end{aligned}    
    \end{equation}
    \end{enumerate}
    where $x = (x_1, y) \in \mathbb{R} \times \mathbb{R}^{d - 1} = \mathbb{R}^d$.
    \begin{proof}
    We first prove \eqref{L-H interation}. Since linear Schr\"{o}dinger and wave equations are translation-invariant and the operators $P_{\lambda}, P_{\mu, e_1}$ commute with the usual derivative $\partial$, by Proposition \ref{prop4.2} we have
    \begin{equation}\label{d-c1}
        \begin{aligned}
            &\partial_{t}\int_{\mathbb{R}^{d-1}} e(P_{\lambda}n)dy-\partial_{x_1}\int_{\mathbb{R}^{d-1}}\partial_{t}P_{\lambda}n\partial_{x_1} P_{\lambda}ndy =0,\\
            &\partial_{t}\int_{\mathbb{R}^{d-1}} P_{\mu,e_{1}}E\overline{P_{\mu,e_{1}}E^a} dy-\partial_{x_1}\int_{\mathbb{R}^{d-1}}\Im(P_{\mu,e_{1}}E\overline{\partial_{x_1}P_{\mu,e_{1}}E^a})dy=0,
        \end{aligned}
    \end{equation}
    where $E^a(x) := E(x - a), a \in \mathbb{R}^d$.
    Applying Lemma \ref{div-curl} to \eqref{d-c1} shows 
    \begin{equation}\label{dc}
        {\rm LHS} \lesssim {\rm RHS},
    \end{equation}
    where 
    \begin{equation*}
    \begin{aligned}
        {\rm LHS} &:= \int_0^T\int_{\mathbb{R}}\left(\int_{\mathbb{R}^{d-1}} e(P_{\lambda}n)dy\right)\left(\int_{\mathbb{R}^{d-1}}\Im(\overline{P_{\mu,e_{1}}E}\partial_{x_1}P_{\mu,e_{1}}E^a)dy\right) dx_1dt\\
        &- \int_0^T\int_{\mathbb{R}}\left(\int_{\mathbb{R}^{d-1}}\partial_{t}P_{\lambda}n\partial_{x_1} P_{\lambda}n\right)\left(\int_{\mathbb{R}^{d-1}} \overline{P_{\mu,e_{1}}E}P_{\mu,e_{1}}E^a dy\right) dx_1dt,\\ 
        {\rm RHS} &:= \left(\Vert e(P_{\lambda}n)(0)\Vert_{L_x^1} + \Vert e(P_{\lambda}n)\Vert_{L_t^{\infty}L_x^1}\right)\left(\Vert\overline{P_{\mu,e_{1}}E(0)}P_{\mu,e_{1}}E^a(0)\Vert_{L_x^1} + \Vert\overline{P_{\mu,e_{1}}E}P_{\mu,e_{1}}E^a\Vert_{L_t^{\infty}L_x^1}\right).
    \end{aligned}
\end{equation*}
    Take supremum over $a$ on both sides of \eqref{dc}. Note that (see \cite[Lemma 4.3]{tu2025physicalspaceproofbilinear})
    \begin{equation*}
        \mu\Vert P_{\mu, e_1}E\Vert_{L_{t, x}^2}^2 \lesssim \sup_a\int_0^T\int_{\mathbb{R}}\left(\int_{\mathbb{R}^{d-1}}\Im(P_{\mu,e_{1}}E\overline{\partial_{x_1}P_{\mu,e_{1}}E^a})dy\right) dx_1dt
    \end{equation*}
    and also 
    \begin{equation*}
        \partial_tP_{\lambda}n \cdot \partial_{x_1}P_{\lambda}n \lesssim |\partial_tP_{\lambda}n|^2 + |\partial_{x_1}P_{\lambda}n|^2 \lesssim e(P_{\lambda}n).
    \end{equation*}
    We have 
    \begin{equation}\label{84}
    \begin{aligned}
        \sup_a{\rm LHS} \geq &\sup_a \int_0^T\int_{\mathbb{R}}\left(\int_{\mathbb{R}^{d-1}} e(P_{\lambda}n)dy\right)\left(\int_{\mathbb{R}^{d-1}}\Im(P_{\mu,e_{1}}E\overline{\partial_{x_1}P_{\mu,e_{1}}E^a})dy\right) dx_1dt\\
        - &\sup_a\int_0^T\int_{\mathbb{R}}\left(\int_{\mathbb{R}^{d-1}}\partial_{t}P_{\lambda}n\partial_{x_1} P_{\lambda}n\right)\left(\int_{\mathbb{R}^{d-1}} P_{\mu,e_{1}}E\overline{P_{\mu,e_{1}}E^a} dy\right) dx_1dt \\ 
        \gtrsim &(\mu - 1)\int_0^T\int_{\mathbb{R}}\left(\int_{\mathbb{R}^{d-1}} e(P_{\lambda}n)dy\right)\left(\int_{\mathbb{R}^{d - 1}}|P_{\mu, e_1}E|^2 dy\right) dx_1dt \\ 
        \gtrsim &\mu\int_0^T\int_{\mathbb{R}}\left(\int_{\mathbb{R}^{d-1}} e(P_{\lambda}n)dy\right)\left(\int_{\mathbb{R}^{d - 1}}|P_{\mu, e_1}E|^2 dy\right) dx_1dt \\ 
        \gtrsim &\mu\int_0^T\int_{\mathbb{R}}\left(\int_{\mathbb{R}^{d-1}} |\nabla_xP_{\lambda}n|^2dy+\int_{\mathbb{R}^{d-1}} |\partial_{t}P_{\lambda}n|^2dy\right)\left(\int_{\mathbb{R}^{d - 1}}|P_{\mu, e_1}E|^2 dy\right) dx_1dt \\ 
        \gtrsim &\mu\lambda^2 \Big\Vert \Vert P_{\lambda}n\Vert_{L_{y}^{2}}\Vert P_{\mu, e_1}E\Vert_{L_{y}^{2}}\Big\Vert_{L_{t,x_1}^{2}}^2+\mu \Big\Vert \Vert \partial_{t}P_{\lambda}n\Vert_{L_{y}^{2}}\Vert P_{\mu, e_1}E\Vert_{L_{y}^{2}}\Big\Vert_{L_{t,x_1}^{2}}^2
    \end{aligned}
\end{equation}
    since $\mu \gg 1$. Moreover, from the conservation laws 
    \begin{equation}\label{conservation}
        \Vert e(P_{\lambda}n)\Vert_{L_{t}^{\infty}L_{x}^{1}}=\Vert e(P_{\lambda}n)(0)\Vert_{L_{x}^{1}}, \quad \Vert P_{\mu,e_{1}}E\Vert_{L_{t}^{\infty}L^{2}_{x}}=\Vert P_{\mu,e_{1}}E(0)\Vert_{L^{2}_{x}}
    \end{equation}
    and H\"older's inequality we obtain 
    \begin{equation}\label{86}
        \begin{aligned}
            \sup_a{\rm RHS} \lesssim \Vert e(P_{\lambda}n)(0)\Vert_{L_{x}^{1}}\Vert P_{\mu, e_1}E(0)\Vert_{L_x^2}^2 \lesssim \Vert P_{\mu, e_1}E_0\Vert_{L_x^2}^2\left(\lambda^2\Vert P_{\lambda}n_0\Vert_{L_x^2}^2 + \Vert P_{\lambda}n_1\Vert_{L_x^2}^2\right).
        \end{aligned}
    \end{equation}
    Now \eqref{L-H interation} follows from \eqref{dc}, \eqref{84} and \eqref{86}.
    
    For \eqref{H-L interaction}, from Proposition \ref{prop4.2} we have 
    \begin{equation}\label{n high}
        \begin{aligned}
            &\partial_{t}\int_{\mathbb{R}^{d-1}} \frac{\vert P_{\mu}E\vert^{2}}{2}dy-\partial_{x_1}\int_{\mathbb{R}^{d-1}} \Im(P_{\mu}E\overline{\partial_{x_1}P_{\mu}E}) dy=0, \\
            &\partial_{t}\int_{\mathbb{R}^{d-1}} \partial_{t}P_{\lambda,e_{1}}n^b\partial_{x_1}P_{\lambda,e_{1}}n^bdy - \partial_{x_1}\int_{\mathbb{R}^{d-1}} \frac{1}{2}(\vert \partial_{t}P_{\lambda,e_{1}}n^b\vert^{2}+\vert \partial_{x_1}P_{\lambda,e_{1}}n^b\vert^{2}-\vert \nabla_{y}P_{\lambda,e_{1}}n^b\vert^{2})dy =0.
        \end{aligned}
    \end{equation}
    Similarly, by Lemma \ref{div-curl} and conservation laws \eqref{conservation} we obtain
    \begin{equation}\label{88}
        \begin{aligned}
            &\sup_b\int_0^T\int_{\mathbb{R}}\left(\int_{\mathbb{R}^{d - 1}}e(P_{\lambda,e_{1}}n^b)dy\right)\left(\int_{\mathbb{R}^{d-1}} \vert P_{\mu}E\vert^{2}dy\right)dx_1dt \\ 
            - &\sup_b\int_0^T\int_{\mathbb{R}}\left(\int_{\mathbb{R}^{d-1}}\vert \nabla_{y}P_{\lambda,e_{1}}n^b\vert^{2}dy\right)\left(\int_{\mathbb{R}^{d-1}} \vert P_{\mu}E\vert^{2}dy\right) dx_1dt \\
            - &\sup_b\int_0^T\int_{\mathbb{R}}\left(\int_{\mathbb{R}^{d-1}} \partial_{t}P_{\lambda,e_{1}}n^b\partial_{x_1}P_{\lambda,e_{1}}n^b dy\right)\left(\int_{\mathbb{R}^{d-1}} \Im(P_{\mu}E\overline{\partial_{x_1}P_{\mu}E})dy\right) dx_1dt \\ 
            \lesssim &\sup_b\Vert e(P_{\lambda, e_1}n^b)(0)\Vert_{L_x^1}\Vert P_{\mu}E(0)\Vert_{L_x^2}^2 \\ 
            \lesssim &\Vert P_{\mu}E_0\Vert_{L_x^2}^2\left(\lambda^2\Vert P_{\lambda, e_1}n_0\Vert_{L_x^2}^2 + \Vert P_{\lambda, e_1}n_1\Vert_{L_x^2}^2\right).
        \end{aligned}
    \end{equation}
    We first assume that $\lambda = 1$. Recall the definition of $P_{1, e_1}$. We have 
    \begin{equation}\label{82}
    \begin{aligned}
        &\sup_b\int_0^T\int_{\mathbb{R}}\left(\int_{\mathbb{R}^{d-1}}\vert \nabla_{y}P_{1,e_{1}}n^b\vert^{2}dy\right)\left(\int_{\mathbb{R}^{d-1}} \vert P_{\mu}E\vert^{2}dy\right) dx_1dt \\
        \lesssim &\left(\frac{1}{100}\lambda\right)^2\Big\Vert \Vert P_{1,e_{1}}n\Vert_{L_{y}^{2}}\Vert P_{\mu}E\Vert_{L_{y}^{2}}\Big\Vert_{L_{t,x_1}^{2}}^2 \\
        = &\left(\frac{1}{100}\right)^2\Big\Vert \Vert P_{1,e_{1}}n\Vert_{L_{y}^{2}}\Vert P_{\mu}E\Vert_{L_{y}^{2}}\Big\Vert_{L_{t,x_1}^{2}}^2 
    \end{aligned}
\end{equation}
and also 
\begin{equation}\label{83}
    \begin{aligned}
                &\sup_b\int_0^T\int_{\mathbb{R}}\left(\int_{\mathbb{R}^{d - 1}}e(P_{1,e_{1}}n^b)dy\right)\left(\int_{\mathbb{R}^{d-1}} \vert P_{\mu}E\vert^{2}dy\right)dx_1dt \\
        \gtrsim &\Big\Vert \Vert P_{1,e_{1}}n\Vert_{L_{y}^{2}}\Vert P_{\mu}E\Vert_{L_{y}^{2}}\Big\Vert_{L_{t,x_1}^{2}}^2 + \Big\Vert \Vert \partial_tP_{1,e_{1}}n\Vert_{L_{y}^{2}}\Vert P_{\mu}E\Vert_{L_{y}^{2}}\Big\Vert_{L_{t,x_1}^{2}}^2, \\
                &\sup_b\int_0^T\int_{\mathbb{R}}\left(\int_{\mathbb{R}^{d-1}} \partial_{t}P_{1,e_{1}}n^b\partial_{x_1}P_{1,e_{1}}n^b dy\right)\left(\int_{\mathbb{R}^{d-1}} \Im(P_{\mu}E\overline{\partial_{x_1}P_{\mu}E})dy\right) dx_1dt \\ 
        \lesssim &\mu\left(\Big\Vert \Vert P_{1,e_{1}}n\Vert_{L_{y}^{2}}\Vert P_{\mu}E\Vert_{L_{y}^{2}}\Big\Vert_{L_{t,x_1}^{2}}^2 + \Big\Vert \Vert \partial_tP_{1,e_{1}}n\Vert_{L_{y}^{2}}\Vert P_{\mu}E\Vert_{L_{y}^{2}}\Big\Vert_{L_{t,x_1}^{2}}^2\right) \\ 
        \ll &\Big\Vert \Vert P_{1,e_{1}}n\Vert_{L_{y}^{2}}\Vert P_{\mu}E\Vert_{L_{y}^{2}}\Big\Vert_{L_{t,x_1}^{2}}^2 + \Big\Vert \Vert \partial_tP_{\lambda,e_{1}}n\Vert_{L_{y}^{2}}\Vert P_{\mu}E\Vert_{L_{y}^{2}}\Big\Vert_{L_{t,x_1}^{2}}^2.
    \end{aligned}
\end{equation}
It follows from \eqref{88}-\eqref{83} that 
\begin{equation*}
    \begin{aligned}
        \Big\Vert \Vert e(P_{1,e_{1}}n)\Vert_{L_{y}^{1}}\Vert P_{\mu}E\Vert_{L_{y}^{2}}^2\Big\Vert_{L_{t,x_1}^1} &= \Big\Vert \Vert P_{1,e_{1}}n\Vert_{L_{y}^{2}}\Vert P_{\mu}E\Vert_{L_{y}^{2}}\Big\Vert_{L_{t,x_1}^{2}}^2 + \Big\Vert \Vert \partial_tP_{1,e_{1}}n\Vert_{L_{y}^{2}}\Vert P_{\mu}E\Vert_{L_{y}^{2}}\Big\Vert_{L_{t,x_1}^{2}}^2 \\
    &\lesssim \Vert P_{\mu}E_0\Vert_{L_x^2}^2\left(\Vert P_{1, e_1}n_0\Vert_{L_x^2} + \Vert P_{1, e_1}n_1\Vert_{L_x^2}^2\right).
    \end{aligned}
\end{equation*}
In general, we have
\begin{equation*}
    \begin{aligned}
        &\lambda^2\Big\Vert \Vert P_{\lambda,e_{1}}n\Vert_{L_{y}^{2}}\Vert P_{\mu}E\Vert_{L_{y}^{2}}\Big\Vert_{L_{t,x_1}^{2}}^2 + \Big\Vert \Vert \partial_tP_{\lambda,e_{1}}n\Vert_{L_{y}^{2}}\Vert P_{\mu}E\Vert_{L_{y}^{2}}\Big\Vert_{L_{t,x_1}^{2}}^2 \\ 
        \sim &\Big\Vert \Vert e(P_{\lambda,e_{1}}n)\Vert_{L_{y}^{1}}\Vert P_{\mu}E\Vert_{L_{y}^{2}}^2\Big\Vert_{L_{t,x_1}^1} \\ 
        = &\Big\Vert \Vert e(\lambda^{-d}P_{1,e_{1}}n(t, \lambda^{-1}x_1, \lambda^{-1}\cdot))\Vert_{L_{y}^{1}}\Vert P_{\mu}E\Vert_{L_{y}^{2}}^2\Big\Vert_{L_{t,x_1}^1} \\
        = &\lambda^{-2d - 1}\Big\Vert \Vert e(P_{1,e_{1}}n(\lambda^{-1}t, \lambda^{-1}x_1, \lambda^{-1}\cdot))\Vert_{L_{y}^{1}}\Vert P_{\mu}E(\lambda^{-1}t, x_1, \cdot)\Vert_{L_{y}^{2}}^2\Big\Vert_{L_{t,x_1}^1} \\
        \lesssim &\lambda^{-1}\Vert e(\lambda^{-d}P_{1, e_1}n)(0)\Vert_{L_x^1}\Vert P_{\mu}E_0\Vert_{L_x^2}^2 \\ 
        \lesssim &\lambda^{-1}\Vert e(P_{\lambda, e_1}n)(0)\Vert_{L_x^1}\Vert P_{\mu}E_0\Vert_{L_x^2}^2 \\ 
        = &\lambda^{-1}\Vert P_{\mu}E_0\Vert_{L_x^2}^2\left(\lambda^2\Vert P_{\lambda, e_1}n_0\Vert_{L_x^2}^2 + \Vert P_{\lambda, e_1}n_1\Vert_{L_x^2}^2\right).
    \end{aligned}
\end{equation*}
This completes the proof.
\end{proof}
\end{theorem}

\begin{corollary}\label{coro}
  For $E, v$ satisfying 
  \begin{equation*}
      \begin{cases}
          \mathrm{i}\partial_{t}E+\Delta E=0,\\
        \mathrm{i}\partial_{t}v-\Lambda v=0, \\ 
        t = 0\colon E = E_0, v = v_0,
      \end{cases}
  \end{equation*}
  there holds
  \begin{enumerate}
    \item ($1\ll \mu$)
    \begin{equation}\label{High E}
        \Vert P_{\leq \lambda} v P_{\mu,e_1}E\Vert_{L_{t}^{2}L_{x_1}^{2}L_{y}^{1}}\lesssim \mu^{-\frac{1}{2}}\Vert  P_{\mu,e_1} E_{0}\Vert_{L_{x}^{2}}\Vert P_{\leq \lambda}v_{0}\Vert_{L_{x}^{2}};
    \end{equation}
    \item ($\mu\ll \lambda$)
        \begin{equation}\label{High W}
        \Vert P_{\lambda,e_1} v P_{\leq \mu}E\Vert_{L_{t}^{2}L_{x_1}^{2}L_{y}^{1}}\lesssim \lambda^{-\frac{1}{2}}\Vert P_{\leq \mu} E_{0}\Vert_{L_{x}^{2}}\Vert  P_{\lambda,e_1}v_{0}\Vert_{L_{x}^{2}}. 
    \end{equation} 
  \end{enumerate}  
\end{corollary}
\begin{proof}
    Note that $P_{\leq \lambda}$ also commutes with the usual derivative $\partial$.
    Since $n = \Re v$ is a solution to equation $\partial_t^2n - \Delta n = 0$, when $1 \ll \mu$, from H\"older's inequality and Theorem \ref{theorem H-L H-H} we have 
    \begin{equation*}
    \begin{aligned}
        \Vert P_{\leq \lambda} v P_{\mu,e_1}E\Vert_{L_{t}^{2}L_{x_1}^{2}L_{y}^{1}} &\leq \Big\Vert \Vert P_{\leq \lambda} v\Vert_{L_{y}^{2}}\Vert  P_{\mu,e_{1}}E\Vert_{L_{y}^{2}}\Big\Vert_{L_{t,x_1}^{2}} \\
        &\lesssim  \Big\Vert \Vert P_{\leq \lambda} \Re v\Vert_{L_{y}^{2}}\Vert  P_{\mu,e_{1}}E\Vert_{L_{y}^{2}}\Big\Vert_{L_{t,x_1}^{2}} +\Big\Vert \Vert P_{\leq \lambda}\Im v\Vert_{L_{y}^{2}}\Vert  P_{\mu,e_{1}}E\Vert_{L_{y}^{2}}\Big\Vert_{L_{t,x_1}^{2}}               \\
        &=  \Big\Vert \Vert P_{\leq \lambda} \Re v\Vert_{L_{y}^{2}}\Vert  P_{\mu,e_{1}}E\Vert_{L_{y}^{2}}\Big\Vert_{L_{t,x_1}^{2}} +\Big\Vert \Vert P_{\leq \lambda}\Lambda^{-1}\partial_{t}\Re v\Vert_{L_{y}^{2}}\Vert  P_{\mu,e_{1}}E\Vert_{L_{y}^{2}}\Big\Vert_{L_{t,x_1}^{2}}               \\
        &\lesssim \mu^{-\frac{1}{2}}\Vert  P_{\mu,e_{1}}E_{0}\Vert_{L_{x}^{2}}\left(\Vert P_{\leq \lambda}\Re v_{0}\Vert_{L_{x}^{2}}+\Vert P_{\leq \lambda}\Im v_{0}\Vert_{L_{x}^{2}}\right)\\
        &\lesssim\mu^{-\frac{1}{2}}\Vert  P_{\mu,e_{1}}E_{0}\Vert_{L_{x}^{2}}\Vert P_{\leq \lambda}v_{0}\Vert_{L_{x}^{2}}.
    \end{aligned}
\end{equation*}
The proof for \eqref{High W} follows similarly.
\end{proof}

\section{Proof of Proposition \ref{prop} when \texorpdfstring{$d = 2$}{d=2}}\label{d=2}

In this section, we aim to prove the convergence of the iteration sequences $\{E^{(k)}\}_{k = 0}^{\infty}$ and $\{v^{(k)}\}_{k = 0}^{\infty}$ defined in \eqref{iteration} by induction.
We begin with the two-dimensional case.

\subsection{The estimates for linear equations}\label{linear estimate}

Recall the iteration space $S, N$ and $X$ we defined in Section \ref{Function Space}.
By \eqref{stricahrtz d=2} and Strichartz estimates \eqref{AB}-\eqref{CD} we have 
\begin{gather}
    \label{E_linear}\Vert E^{(0)}\Vert_{X(T)} \lesssim \Vert E_0\Vert_{L_x^2}, \\ 
    \Vert v^{(0)}\Vert_{S_2(T)} \lesssim \Vert v_0\Vert_{H^{-\frac{1}{2}}_x}.
\end{gather}

\subsection{The first step}\label{d=2,k=1}

To bound the term $\Vert E^{(1)}\Vert_{X(T)}$, by Strichartz estimates \eqref{stricahrtz d=2} and \eqref{AB} we have 
\begin{equation*}
    \Vert E^{(1)}\Vert_{X(T)} \lesssim \Vert E_0\Vert_{L_x^2} + \Vert (\Re v^{(0)})E^{(0)}\Vert_{N_1(T)}\leq \Vert E_0\Vert_{L_x^2} + \Vert  v^{(0)}E^{(0)}\Vert_{N_1(T)}.
\end{equation*}
Here we use dyadic decomposition and bilinear estimates (Corollary \ref{coro}) to control $\Vert  v^{(0)}E^{(0)}\Vert_{N_1(T)}$.

Firstly, by duality we have  
\begin{equation*}
    \begin{aligned}
        \Vert  v^{(0)}E^{(0)}\Vert_{N_1(T)} = \sup_{G \in S_1(T), \Vert G\Vert_{S_1(T)} = 1}\int \overline{G} v^{(0)}E^{(0)} dxdt.
    \end{aligned}
\end{equation*}
For such $G$,  
    \begin{equation*}
            \begin{aligned}
                \int \overline{G} v^{(0)}E^{(0)} dxdt = \sum_{\sigma \geq 1}\int\overline{G}P_{\sigma}\left( v^{(0)}E^{(0)}\right) dxdt.
    \end{aligned}
\end{equation*}
Note that for fixed $\sigma$,  
\begin{equation*}
            \int\overline{G}P_{\sigma}\left( v^{(0)}E^{(0)}\right) dx = \int \overline{\mathcal{F}_xG}\mathcal{F}_xP_{\sigma}\left( v^{(0)}E^{(0)}\right) d\xi.
        \end{equation*}
Hence the right-hand side vanishes unless the support of $G$ is contained in $\{|\xi| \sim \sigma\}$.
It follows that
        \begin{equation}\label{integral G}
            \begin{aligned}
                \int \overline{G} v^{(0)}E^{(0)} dxdt &= \left(\sum_{1\leq\lambda\ll \mu\sim\sigma}+\sum_{1\leq\mu\ll \lambda\sim\sigma}+\sum_{1\leq\sigma\lesssim\lambda\sim \mu}\right)\int\overline{P_{\sigma}G}P_{\lambda} v^{(0)}P_{\mu}E^{(0)} dxdt \\ 
                &=: I_1 + I_2 + I_3. 
            \end{aligned}
\end{equation} 
In the remaining part of Section \ref{d=2,k=1}, we denote by $E=E^{(0)}$ and $v=v^{(0)}$. 
\begin{proposition}\label{lemma l2}
    There holds 
    \begin{equation*}
        I_1, I_2, I_3 \lesssim T^{\frac{1}{4}}\Vert E_0\Vert_{L_x^2}\Vert v_0\Vert_{H_x^{-\frac{1}{2}}}.
    \end{equation*}
    \begin{proof}
        By H\"{o}lder's inequality and bilinear estimate \eqref{High E},
\begin{align*}
\begin{split}
    I_{1} &\leq \sum_{1\ll\mu\sim\sigma}\sum_{i = 1}^N\int \overline{  P_{\sigma}G}P_{\ll \mu} vP_{\mu,\omega_{i}}Edxdt\\
    &\leq \sum_{1\ll\mu\sim\sigma}\sum_{i = 1}^N\Vert  P_{\sigma}G\Vert_{L_{t}^{4}L_{\omega_{i}}^{2}L_{\omega_{i\bot}}^{\infty}}\Vert P_{\ll \mu} v P_{\mu, \omega_{i}}E\Vert_{L_{t}^{\frac{4}{3}}L_{\omega_{i}}^{2}L_{\omega_{i\bot}}^{1}} \\
    &\leq T^{\frac{1}{4}}\sum_{1\ll\mu\sim\sigma}\sum_{i = 1}^N\Vert  P_{\sigma}G\Vert_{L_{t}^{4}L_{\omega_{i}}^{2}L_{\omega_{i\bot}}^{\infty}}\Vert P_{\ll \mu} v P_{\mu, \omega_{i}}E\Vert_{L_{t}^{2}L_{\omega_{i}}^{2}L_{\omega_{i\bot}}^{1}} \\
    &\lesssim T^{\frac{1}{4}}\sum_{1\ll\mu\sim\sigma}\sum_{i = 1}^N\mu^{-\frac{1}{2}}\Vert  P_{\sigma}G\Vert_{L_{t}^{4}L_{\omega_{i}}^{2}L_{\omega_{i\bot}}^{\infty}}\Vert P_{\ll \mu}v_0\Vert_{L_x^2} \Vert P_{\mu, \omega_{i}}E_0\Vert_{L_x^2} \\
    &\leq T^{\frac{1}{4}}\Vert v_0\Vert_{H_x^{-\frac{1}{2}}}\sum_{1\ll\mu\sim\sigma}\sum_{i = 1}^N\Vert  P_{\sigma}G\Vert_{L_{t}^{4}L_{\omega_{i}}^{2}L_{\omega_{i\bot}}^{\infty}}\Vert P_{\mu, \omega_{i}}E_0\Vert_{L_x^2} \\
    &\leq T^{\frac{1}{4}}\Vert v_0\Vert_{H_x^{-\frac{1}{2}}}\left(\sum_{\sigma \geq 1}\sum_{i = 1}^N\Vert  P_{\sigma}G\Vert^2_{L_{t}^{4}L_{\omega_{i}}^{2}L_{\omega_{i\bot}}^{\infty}}\right)^{\frac{1}{2}}\left(\sum_{\sigma \geq 1}\sum_{i = 1}^N\Vert P_{\sigma,\omega_{i}}E_0\Vert_{L_x^2}^2\right)^{\frac{1}{2}} \\
    &\lesssim T^{\frac{1}{4}}\Vert v_0\Vert_{H_x^{-\frac{1}{2}}}\Vert G\Vert_{S_1(T)}\Vert E_0\Vert_{L_x^2} \\
    &=T^{\frac{1}{4}}\Vert v_0\Vert_{H_x^{-\frac{1}{2}}}\Vert E_0\Vert_{L_x^2}.
\end{split}  
\end{align*}

Similarly, for $I_3$ we have
\begin{equation*}
    \begin{aligned}
        I_3 &\lesssim \sum_{1 \lesssim \lambda \sim \mu}\sum_{i = 1}^N\int\overline{P_{\lesssim \mu}G}P_{\lambda} vP_{\mu, \omega_i}E dxdt \\ 
        &\lesssim \sum_{1 \lesssim \lambda \sim \mu}\sum_{i = 1}^N\Vert P_{\lesssim \mu}G\Vert_{L_t^4L_{\omega_i}^2L_{\omega_{i\bot}}^{\infty}} \Vert P_{\lambda} vP_{\mu, \omega_i}E\Vert_{L_t^{\frac{4}{3}}L_{\omega_i}^2L_{\omega_{i\bot}}^1} \\ 
        &\lesssim T^{\frac{1}{4}}\Vert G\Vert_{S_1(T)}\sum_{1 \lesssim \lambda \sim \mu}\sum_{i= 1}^N\Vert P_{\lambda} vP_{\mu, \omega_i}E\Vert_{L_t^2L_{\omega_i}^2L_{\omega_{i\bot}}^1} \\ 
        &\lesssim T^{\frac{1}{4}}\sum_{1 \lesssim \lambda \sim \mu}\sum_{i = 1}^N\mu^{-\frac{1}{2}}\Vert P_{\lambda}v_0\Vert_{L_x^2}\Vert P_{\mu, \omega_i}E_0\Vert_{L_x^2} \\ 
        &\lesssim T^{\frac{1}{4}}\left(\sum_{\mu \geq 1}\mu^{-1}\Vert P_{\mu}v_0\Vert_{L_x^2}^2\right)^{\frac{1}{2}}\left(\sum_{\mu \geq 1}\sum_{i = 1}^N\Vert P_{\mu, \omega_i}E_0\Vert_{L_x^2}^2\right)^{\frac{1}{2}} \\ 
        &\lesssim T^{\frac{1}{4}}\Vert v_0\Vert_{H_x^{-\frac{1}{2}}}\Vert E_0\Vert_{L_x^2}.
    \end{aligned}
\end{equation*}

Finally, for $I_2$, 
\begin{equation*}
    \begin{split}
        I_{2} &\lesssim \sum_{1\ll \lambda\sim\sigma }\sum_{i = 1}^{N}\int \overline{ P_{\sigma}G}P_{\lambda,\omega_{i}} vP_{\ll\lambda}E dxdt\\
        &\leq \sum_{1\ll \lambda\sim\sigma}\sum_{i = 1}^N\Vert  P_{\sigma}G\Vert_{L_{t}^{4}L_{\omega_{i}}^{2}L_{\omega_{i\bot}}^{\infty}}\Vert  P_{\lambda,\omega_{i}} v P_{\ll \lambda}E\Vert_{L_{t}^{\frac{4}{3}}L_{\omega_{i}}^{2}L_{\omega_{i\bot}}^{1}}\\
        &\leq T^{\frac{1}{4}}\sum_{1\ll \lambda\sim\sigma}\sum_{i = 1}^N\Vert  P_{\sigma}G\Vert_{L_{t}^{4}L_{\omega_{i}}^{2}L_{\omega_{i\bot}}^{\infty}}\Vert  P_{\lambda,\omega_{i}} v P_{\ll \lambda}E\Vert_{L_{t}^2L_{\omega_{i}}^{2}L_{\omega_{i\bot}}^{1}}\\
        &\lesssim T^{\frac{1}{4}}\sum_{\substack{1\ll \lambda\sim\sigma}}\sum_{i = 1}^N\lambda^{-\frac{1}{2}}\Vert  P_{\sigma}G\Vert_{L_{t}^{4}L_{\omega_{i}}^{2}L_{\omega_{i\bot}}^{\infty}}\Vert P_{\lambda,\omega_{i}} v_{0}\Vert_{L_{x}^{2}}\Vert P_{\ll\lambda}E_{0}\Vert_{L_{x}^{2}}\\
        &\lesssim T^{\frac{1}{4}}\Vert E_{0}\Vert_{L_{x}^{2}}\left( \sum_{\substack{\sigma \geq 1}}\sum_{i = 1}^N\Vert  P_{\sigma}G\Vert_{L_{t}^{4}L_{\omega_{i}}^{2}L_{\omega_{i\bot}}^{\infty}}^{2}\right)^{\frac{1}{2}}\left( \sum_{\substack{\sigma\geq 1}}\sum_{i=1}^{N}\sigma^{-1}\Vert P_{\sigma,\omega_{i}} v_{0}\Vert_{L_{x}^{2}}^{2}\right)^{\frac{1}{2}}\\
        &\lesssim T^{\frac{1}{4}}\Vert v_{0}\Vert_{H_x^{-\frac{1}{2}}}\Vert E_{0}\Vert_{L_{x}^{2}}\Vert G\Vert_{S_{1}(T)}\\
        &=T^{\frac{1}{4}}\Vert v_{0}\Vert_{H_x^{-\frac{1}{2}}}\Vert E_{0}\Vert_{L_{x}^{2}}.
    \end{split}
\end{equation*}
    \end{proof}
\end{proposition}
Now from Proposition \ref{lemma l2} we have
\begin{equation*}
    \Vert E^{(1)}\Vert_{S_{1}(T)}\lesssim\Vert E_{0}\Vert_{L^{2}_x}+T^{\frac{1}{4}}\Vert E_{0}\Vert_{L^{2}_x}\Vert v_{0}\Vert_{H^{-\frac{1}{2}}_x}.
\end{equation*}

For the estimate of $\Vert v^{(1)}\Vert_{S_{2}(T)}$, we first write
\begin{equation}\label{100}
    \Vert v^{(1)}\Vert_{S_{2}(T)}=\sup_{K \in S^{*}_2(T), \Vert K\Vert_{S^{*}_2(T)} = 1}\int Kv^{(1)} dxdt,
\end{equation}
where $S^{*}_{2}(T)=L_{t}^{1}H^{1/2}_{x}$. Consider a function $h:[0,T]\times \mathbb{R}^{2}\rightarrow \mathbb{C}$ satisfying 
\begin{equation}\label{101}
    \begin{cases}
        \mathrm{i}\partial_{t}h+\Lambda h=K,\\
        t=0:h=0.
    \end{cases}
\end{equation}
Then we have
\begin{equation*}
    \Vert h\Vert_{L_t^{\infty}H_{x}^{\frac{1}{2}}}\lesssim \Vert K\Vert_{S_{2}^{*}(T)}. 
\end{equation*}
Note that
\begin{equation}\label{integral K}
\begin{split}
    \int Kv^{(1)} dxdt=\int \left(\mathrm{i}\partial_{t}h+\Lambda h\right)v^{(1)} dxdt=-\int h(\ii\partial_{t}-\Lambda)v^{(1)}dxdt&=-\int h\Lambda |E^{(0)}|^{2}dxdt\\
    &=\int (\Lambda h)| E^{(0)}|^{2}dxdt
\end{split}
\end{equation}
and $h\in L_{t}^{\infty}H_{x}^{\frac{1}{2}}$, $\Lambda h\in L_{t}^{\infty}H^{-\frac{1}{2}}_{x}=S_{2}(T)$.
Applying the div-curl lemma to the following balance laws:
\begin{align*}
    &\partial_t\int_{\mathbb{R}}\frac{|P_{\leq\mu}E^{(0)}|^2}{2} dy - \partial_{x_1}\int_{\mathbb{R}}\Im\left(P_{\leq\mu}E^{(0)}\overline{\partial_{x_1}P_{\leq\mu}E^{(0)}}\right) dy = 0, \\ 
            &\partial_t\int_{\mathbb{R}}\Im(\Lambda P_{\lambda, e_1}h)\partial_{x_1} P_{\lambda, e_1}(\Re h) dy - \partial_{x_1}\int_{\mathbb{R}}\left(|\Lambda P_{\lambda, e_1} h|^2 - |\nabla_yP_{\lambda, e_1}(\Re h)|^2\right) dy \\
        = &\int_{\mathbb{R}}P_{\lambda, e_1}\partial_{x_1}(\Re h)P_{\lambda, e_1}Kdy
\end{align*}
and 
\begin{equation}\label{105}
    \begin{aligned}
        &\partial_t\int_{\mathbb{R}}|\Lambda P_{\leq\lambda}h| dy - \partial_{x_1}\int_{\mathbb{R}}\Im(\Lambda P_{\leq\lambda}h)\partial_{x_1} P_{\leq\lambda}(\Re h) dy = \int_{\mathbb{R}}P_{\leq\lambda}KP_{\leq\lambda}\Im(\Lambda h)dy, \\ 
    &\partial_t\int_{\mathbb{R}}\Im\left(P_{\mu, e_1}E^{(0)}\overline{\partial_{x_1}P_{\mu, e_1}E^{(k+1)}}\right) dy - \partial_{x_1}\int_{\mathbb{R}}\left(2|\partial_{x_1}P_{\mu, e_1}E^{(0)}|^2 -\partial_{x_1}^2\frac{|P_{\mu, e_1}E^{(0)}|^2}{2}\right) dy = 0,
    \end{aligned}
\end{equation}
we can also obtain the corresponding bilinear estimates (e.g., Corollary \ref{coro}) and, in a similar fashion to the proof of Proposition \ref{lemma l2}, reach that
\begin{equation*}
    \int \Lambda h | E^{(0)}|^{2} dxdt\lesssim T^{\frac{1}{4}}\Vert K\Vert_{S^{*}_2(T)}\Vert E_{0}\Vert_{L_{x}^{2}}\Vert E^{(0)}\Vert_{S_{1}(T)}\lesssim T^{\frac{1}{4}}\Vert K\Vert_{S_2^*(T)}\Vert E_{0}\Vert_{L_{x}^{2}}^{2}.
\end{equation*}
Therefore, 
\begin{equation*}
    \Vert v^{(1)}\Vert_{S_{2}(T)}\lesssim T^{\frac{1}{4}}\Vert E_{0}\Vert_{L_{x}^{2}}^{2}.
\end{equation*}

To conclude,  
\begin{gather*}
    \Vert E^{(1)}\Vert_{S_{1}(T)}\lesssim\Vert E_{0}\Vert_{L^{2}_x}+T^{\frac{1}{4}}\Vert E_{0}\Vert_{L^{2}_x}\Vert v_{0}\Vert_{H^{-\frac{1}{2}}_x}, \\ 
    \Vert v^{(1)}\Vert_{S_{2}(T)}\lesssim T^{\frac{1}{4}}\Vert E_{0}\Vert_{L_{x}^{2}}^{2}.
\end{gather*}

\subsection{The remaining iteration process}\label{d=2,k geq 2}

When $k \geq 2$, from Strichartz estimates and the definition of $E^{(k)}$, in order to control $\Vert E^{(k)}\Vert_{X(T)}$, we need to bound the term $\Vert (\Re v^{(k - 1)})E^{(k - 1)}\Vert_{N_1(T)}$.
Here we use an inductive argument to bound the iteration sequence $\{E^{(k)}\}$, $\{(\Re v^{(k - 1)})E^{(k - 1)}\}$ and also $\{v^{(k)}\}$, namely, to prove the estimates \eqref{bound}-\eqref{18} in Proposition \ref{prop}.

\begin{proposition}\label{p5.2}
    For all $k \geq 0$ we have 
    \begin{gather}
        \label{108}\Vert E^{(k)}\Vert_{X(T)} \leq C, \\ 
        \label{T* S k}\Vert  v^{(k)}E^{(k)}\Vert_{N_{1}(T)}\leq CT^{\frac{1}{4}},\\
        \label{T* W k}\Vert  v^{(k+1)}E^{(k)}\Vert_{N_{1}(T)}\leq CT^{\frac{1}{4}},
    \end{gather}
    where $C = C\left(\Vert E_0\Vert_{L_x^2}, \Vert v_0\Vert_{H_x^{-1/2}}\right) > 0$. 
    \begin{proof}
        When $k = 0$, by Strichartz estimates
        \begin{equation*}
            \Vert E^{(0)}\Vert_{X(T)} \lesssim \Vert E_0\Vert_{L_x^2},
        \end{equation*}
        and in Section \ref{d=2,k=1} we have obtained that
\begin{equation*}
     \Vert  v^{(0)}E^{(0)}\Vert_{N_{1}(T)}\leq CT^{\frac{1}{4}}\Vert v_{0}\Vert_{H^{-\frac{1}{2}}_x}\Vert E_{0}\Vert_{L^{2}_x}\leq CT^{\frac{1}{4}}.
\end{equation*}
To control the term $\Vert v^{(1)}E^{(0)}\Vert_{N_1(T)}$, we first write 
\begin{equation*}
    \Vert v^{(1)}E^{(0)}\Vert_{N_1(T)} = \sup_{G \in S_1(T), \Vert G\Vert_{S_1(T) = 1}}\int\overline{G}v^{(1)}E^{(0)} dxdt.
\end{equation*}
Recall the definition of $E^{(0)}$ and $v^{(1)}$:
\begin{equation*}
    \begin{cases}
        (\ii\partial_t + \Delta)E^{(0)} = 0, \\ 
        (\ii\partial_t - \Lambda)v^{(1)} = \Lambda|E^{(0)}|^2.
    \end{cases}
\end{equation*} 
It follows that 
\begin{equation}\label{56}
    \begin{aligned}
        &\partial_t\int_{\mathbb{R}}\frac{|P_{\leq\mu}E^{(0)}|^2}{2} dy - \partial_{x_1}\int_{\mathbb{R}}\Im\left(P_{\leq\mu}E^{(0)}\overline{\partial_{x_1}P_{\leq\mu}E^{(0)}}\right) dy = 0, \\ 
                &\partial_t\int_{\mathbb{R}}\Im(\Lambda P_{\lambda, e_1}v^{(1)})\partial_{x_1} P_{\lambda, e_1}(\Re v^{(1)}) dy - \partial_{x_1}\int_{\mathbb{R}}\left(|\Lambda P_{\lambda, e_1} v^{(1)}|^2 - |\nabla_yP_{\lambda, e_1}(\Re v^{(1)})|^2\right) dy \\
        = &\int_{\mathbb{R}}P_{\lambda, e_1}\partial_{x_1}(\Re v^{(1)})P_{\lambda, e_1}(\Delta |E^{(0)}|^{2})dy
    \end{aligned}
    \end{equation}
    and 
    \begin{equation}\label{57}
        \begin{aligned}
        &\partial_t\int_{\mathbb{R}}|\Lambda P_{\leq\lambda}v^{(1)}| dy - \partial_{x_1}\int_{\mathbb{R}}\Im(\Lambda P_{\leq\lambda}v^{(1)})\partial_{x_1} P_{\leq\lambda}(\Re v^{(1)}) dy = \int_{\mathbb{R}}P_{\leq\lambda}(\Delta |E^{(0)}|^{2})P_{\leq\lambda}\Im(\Lambda v^{(1)})dy, \\ 
                &\partial_t\int_{\mathbb{R}}\Im\left(P_{\mu, e_1}E^{(0)}\overline{\partial_{x_1}P_{\mu, e_1}E^{(0)}}\right) dy - \partial_{x_1}\int_{\mathbb{R}}\left(2|\partial_{x_1}P_{\mu, e_1}E^{(0)}|^2 -\partial_{x_1}^2\frac{|P_{\mu, e_1}E^{(0)}|^2}{2}\right) dy = 0,
        \end{aligned}
    \end{equation}
    where $x = (x_1, y) \in \mathbb{R} \times \mathbb{R}$.
    As we have shown in Section \ref{Section Div curl}, applying div-curl lemma to \eqref{56} and \eqref{57}, we have
    \begin{equation}\label{117}
        \begin{aligned}
            &\Vert P_{\leq \lambda}v^{(1)}P_{\mu, e_1}E^{(0)}\Vert_{L_t^2L_{x_1}^2L_y^1} \\
            \lesssim &\mu^{-\frac{1}{2}}\Vert P_{\mu, e_1}E_0\Vert_{L_x^2}\left(\Vert P_{\leq \lambda}v_0\Vert_{L_x^2} + \lambda^{-1}\Vert P_{\leq \lambda}\Im(\Lambda v^{(1)})P_{\leq\lambda}(\Delta |E^{(0)}|^2)\Vert_{L_{t, x}^1}^{\frac{1}{2}}\right) \qquad (1 \ll \mu)
        \end{aligned}
    \end{equation} 
    and 
    \begin{equation}\label{118}
        \begin{aligned}
            &\Vert P_{\lambda, e_1}v^{(1)}P_{\leq \mu}E^{(0)}\Vert_{L_t^2L_{x_1}^2L_y^1} \\
            \lesssim &\lambda^{-\frac{1}{2}}\Vert P_{\leq \mu}E_0\Vert_{L_x^2}\left(\Vert P_{\lambda}v_0\Vert_{L_x^2} + \lambda^{-1}\Vert P_{\lambda, e_1}(\partial_{x_1}\Re v^{(1)})P_{\lambda, e_1}(\Delta |E^{(0)}|^2)\Vert_{L_{t, x}^1}^{\frac{1}{2}}\right) \qquad (\mu \ll \lambda).
        \end{aligned}
    \end{equation}
    Analogous to the proof of Proposition \ref{lemma l2}, by \eqref{117}, \eqref{118} and H\"older's inequality 
    \begin{equation*}
        \begin{aligned}
            \Vert v^{(1)}E^{(0)}\Vert_{N_1(T)} &= \sup_{G \in S_1(T), \Vert G\Vert_{S_1(T) = 1}}\int\overline{G}v^{(1)}E^{(0)} dxdt \\ 
            &\leq C T^{\frac{1}{4}}\Vert E_{0}\Vert_{L^{2}_x}\left(\Vert v_{0}\Vert_{H_x^{-\frac{1}{2}}}+\Vert v^{(1)} |E^{(0)}|^{2}\Vert_{L_{t,x}^{1}}^{\frac{1}{2}}\right)\\
        &\leq C T^{\frac{1}{4}}\Vert E_{0}\Vert_{L^{2}_x}\left(\Vert v_{0}\Vert_{H_x^{-\frac{1}{2}}}+\Vert E^{(0)}\Vert_{S_{1}(T)}^{\frac{1}{2}}\Vert v^{(1)} E^{(0)}\Vert_{N_{1}(T)}^{\frac{1}{2}}\right)\\
        &\leq C T^{\frac{1}{4}}\Vert E_{0}\Vert_{L^{2}_x}\left(\Vert v_{0}\Vert_{H_x^{-\frac{1}{2}}}+\left( C\Vert E_{0}\Vert_{L_{x}^{2}}\right)^{\frac{1}{2}}\Vert v^{(1)} E^{(0)}\Vert_{N_{1}(T)}^{\frac{1}{2}}\right).
        \end{aligned}
    \end{equation*}
    Hence we can choose $T$ sufficiently small such that 
\begin{equation*}
    \Vert  v^{(1)}E^{(0)}\Vert_{N_{1}(T)}\leq CT^{\frac{1}{4}},
\end{equation*}
where $C$ is a constant depending on $\Vert E_0\Vert_{L_x^2}$ and $\Vert v_0\Vert_{H_x^{-1/2}}$.

Now suppose that \eqref{108}-\eqref{T* W k} are valid for $j=k \geq 0$. By direct computations we have   
\begin{equation}\label{123}
    \begin{aligned}
                    &\partial_t\int_{\mathbb{R}}\frac{|P_{\leq\mu}E^{(k+1)}|^2}{2} dy - \partial_{x_1}\int_{\mathbb{R}}\Im\left(P_{\leq\mu}E^{(k+1)}\overline{\partial_{x_1}P_{\leq\mu}E^{(k+1)}}\right) dy \\
            = &\int_{\mathbb{R}}\Im\left(P_{\leq\mu}\left((\Re v^{(k)})E^{(k)}\right)\overline{P_{\leq\mu}E^{(k+1)}}\right)dy, \\ 
            &\partial_t\int_{\mathbb{R}}\Im\left(P_{\mu, e_1}E^{(k+1)}\overline{\partial_{x_1}P_{\mu, e_1}E^{(k+1)}}\right) dy - \partial_{x_1}\int_{\mathbb{R}}\left(2|\partial_{x_1}P_{\mu, e_1}E^{(k+1)}|^2 -\partial_{x_1}^2\frac{|P_{\mu, e_1}E^{(k+1)}|^2}{2}\right) dy \\
            = &\int_{\mathbb{R}}\Im\left(P_{\mu, e_1}\left((\Re v^{(k)})E^{(k)}\right)\overline{\partial_{x_{1}}P_{\mu, e_1}E^{(k+1)}}\right)dy,
    \end{aligned}
    \end{equation}
    and 
    \begin{equation}\label{61}
        \begin{aligned}
                        &\partial_t\int_{\mathbb{R}}|\Lambda P_{\leq\lambda}v^{(k + 1)}| dy - \partial_{x_1}\int_{\mathbb{R}}\Im(\Lambda P_{\leq\lambda}v^{(k+1)})\partial_{x_1} P_{\leq\lambda}(\Re v^{(k+1)}) dy \\
            = &\int_{\mathbb{R}}P_{\leq \lambda}(\Delta |E^{(k)}|^{2})P_{\leq\lambda}\Im(\Lambda v^{(k+1)})dy, \\ 
            &\partial_t\int_{\mathbb{R}}\Im(\Lambda P_{\lambda, e_1}v^{(k+1)})\partial_{x_1} P_{\lambda, e_1}(\Re v^{(k+1)}) dy - \partial_{x_1}\int_{\mathbb{R}}\left(|\Lambda P_{\lambda, e_1} v^{(k+1)}|^2 - |\nabla_yP_{\lambda, e_1}(\Re v^{(k+1)})|^2\right) dy \\
             = &\int_{\mathbb{R}}P_{\lambda, e_1}\partial_{x_1}(\Re v^{(k+1)})P_{\lambda, e_1}(\Delta |E^{(k)}|^{2})dy.
        \end{aligned}
    \end{equation}
    Again, applying div-curl lemma to \eqref{123} and \eqref{61} yields 
        \begin{equation}\label{125}
        \begin{aligned}
            &\Vert P_{\leq \lambda}v^{(k + 1)}P_{\mu, e_1}E^{(k + 1)}\Vert_{L_t^2L_{x_1}^2L_y^1} \\
            \lesssim &\mu^{-\frac{1}{2}}\left(\Vert P_{\mu, e_1}E_0\Vert_{L_x^2} + \Vert P_{\mu, e_1}((\Re v^{(k)})E^{(k)})P_{\mu, e_1}E^{(k)} \Vert_{L_{t, x}^1}^{\frac{1}{2}}\right)\\
            \cdot &\left(\Vert P_{\leq \lambda}v_0\Vert_{L_x^2} + \lambda^{-1}\Vert P_{\leq \lambda}\Im(\Lambda v^{(k)})P_{\leq \lambda}(\Delta |E^{(k)}|^2)\Vert_{L_{t, x}^1}^{\frac{1}{2}}\right)
        \end{aligned}
        \qquad (1 \ll \mu)
    \end{equation} 
    and 
    \begin{equation}\label{126}
        \begin{aligned}
            &\Vert P_{\lambda, e_1}v^{(k + 1)}P_{\leq \mu}E^{(k + 1)}\Vert_{L_t^2L_{x_1}^2L_y^1} \\
            \lesssim &\lambda^{-\frac{1}{2}}\left(\Vert P_{\leq \mu}E_0\Vert_{L_x^2} + \Vert P_{\leq \mu}((\Re v^{(k)})E^{(k)})P_{\leq \mu}E^{(k)} \Vert_{L_{t, x}^1}^{\frac{1}{2}}\right) \\
            \cdot &\left(\Vert P_{\lambda, e_1}v_0\Vert_{L_x^2} + \lambda^{-1}\Vert P_{\lambda, e_1}(\partial_{x_1}\Re v^{(k)})P_{\lambda, e_1}(\Delta |E^{(k)}|^2)\Vert_{L_{t, x}^1}^{\frac{1}{2}}\right)
        \end{aligned}
        \qquad (\mu \ll \lambda).
    \end{equation}
    Similar to the proof of Proposition \ref{lemma l2}, it follows from \eqref{125}, \eqref{126}, H\"older's inequality and Strichartz estimates that  
    \begin{equation}\label{127}
        \begin{aligned}
            &\Vert v^{(k + 1)}E^{(k + 1)}\Vert_{N_1(T)} \\
            \lesssim &T^{\frac{1}{4}}\left(\Vert E_0\Vert_{L_x^2} + \Vert v^{(k)}E^{(k)}E^{(k + 1)}\Vert_{L_{t, x}^1}^{\frac{1}{2}}\right)\left(\Vert v_0\Vert_{H_x^{-1/2}} + \Vert v^{(k + 1)}E^{(k)}E^{(k)}\Vert_{L_{t, x}^1}^{\frac{1}{2}}\right) \\
            \lesssim &T^{\frac{1}{4}}\left(\Vert E_0\Vert_{L_x^2} + \Vert E^{(k + 1)}\Vert_{S_1(T)}^{\frac{1}{2}}\Vert v^{(k)}E^{(k)}\Vert_{N_1(T)}^{\frac{1}{2}}\right)\left(\Vert v_0\Vert_{H_x^{-1/2}} + \Vert E^{(k)}\Vert_{S_1(T)}^{\frac{1}{2}}\Vert v^{(k + 1)}E^{(k)}\Vert_{N_1(T)}^{\frac{1}{2}}\right) \\
            \lesssim &T^{\frac{1}{4}}\left(\Vert E_0\Vert_{L_x^2} + \left(\Vert E_0\Vert_{L_x^2} + \Vert v^{(k)}E^{(k)}\Vert_{N_1(T)}\right)^{\frac{1}{2}}\Vert v^{(k)}E^{(k)}\Vert_{N_1(T)}^{\frac{1}{2}}\right) \\
            \cdot &\left(\Vert v_0\Vert_{H_x^{-1/2}} + \Vert E^{(k)}\Vert_{S_1(T)}^{\frac{1}{2}}\Vert v^{(k + 1)}E^{(k)}\Vert_{N_1(T)}^{\frac{1}{2}}\right).
        \end{aligned}
    \end{equation}
    Repeating the same procedure above for $v^{(k + 2)}$ and $E^{(k + 1)}$, we also have
    \begin{equation}\label{1281}
        \begin{aligned}
            &\Vert v^{(k + 2)}E^{(k + 1)}\Vert_{N_1(T)}\\
            \lesssim &T^{\frac{1}{4}}\left(\Vert E_0\Vert_{L_x^2} + \Vert v^{(k)}E^{(k)}E^{(k + 1)}\Vert_{L_{t, x}^1}^{\frac{1}{2}}\right)\left(\Vert v_0\Vert_{H_x^{-1/2}} + \Vert v^{(k + 2)}E^{(k + 1)}E^{(k + 1)}\Vert_{L_{t, x}^1}^{\frac{1}{2}}\right) \\
            \lesssim &T^{\frac{1}{4}}\left(\Vert E_0\Vert_{L_x^2} + \Vert E^{(k + 1)}\Vert_{S_1(T)}^{\frac{1}{2}}\Vert v^{(k)}E^{(k)}\Vert_{N_1(T)}^{\frac{1}{2}}\right)\left(\Vert v_0\Vert_{H_x^{-1/2}} + \Vert E^{(k + 1)}\Vert_{S_1(T)}^{\frac{1}{2}}\Vert v^{(k + 2)}E^{(k + 1)}\Vert_{N_1(T)}^{\frac{1}{2}}\right) \\
            \lesssim &T^{\frac{1}{4}}\left(\Vert E_0\Vert_{L_x^2} + \left(\Vert E_0\Vert_{L_x^2} + \Vert v^{(k)}E^{(k)}\Vert_{N_1(T)}\right)^{\frac{1}{2}}\Vert v^{(k)}E^{(k)}\Vert_{N_1(T)}^{\frac{1}{2}}\right) \\
            \cdot &\left(\Vert v_0\Vert_{H_x^{-1/2}} + \left(\Vert E_0\Vert_{L_x^2} + \Vert v^{(k)}E^{(k)}\Vert_{N_1(T)}\right)^{\frac{1}{2}}\Vert v^{(k + 2)}E^{(k + 1)}\Vert_{N_1(T)}^{\frac{1}{2}}\right).
        \end{aligned}
    \end{equation}
    From \eqref{127} and \eqref{1281}, it follows that we can choose $T \ll 1$ such that 
    \begin{equation*}
        \Vert v^{(k + 1)}E^{(k + 1)}\Vert_{N_1(T)}, \Vert v^{(k + 2)}E^{(k + 1)}\Vert_{N_1(T)} \leq CT^{\frac{1}{4}},
    \end{equation*}
    and hence \eqref{T* S k} and \eqref{T* W k} hold for all $k \geq 0$ by induction.
    Moreover, by Strichartz estimate 
    \begin{equation*}
        \Vert E^{(k + 1)}\Vert_{X(T)} \leq C\left(\Vert E_0\Vert_{L_x^2} + \Vert (\Re v^{(k)})E^{(k)}\Vert_{N_1(T)} \right) \leq C(1 + CT^{\frac{1}{4}}).
    \end{equation*}
    Therefore, \eqref{108} holds for all $k \geq 0$ by induction. This completes the proof.
    \end{proof}
\end{proposition}

Lastly, to control $\Vert v^{(k + 1)}\Vert_{S_2(T)}\ (k \geq 1)$, from \eqref{100}, \eqref{101} and \eqref{integral K} we write 
\begin{equation}\label{131A}
    \begin{aligned}
        \Vert v^{(k + 1)}\Vert_{S_2(T)} = \sup_{K \in S_2^*(T), \Vert K\Vert_{S_2^*(T)} = 1}\int Kv^{(k + 1)} dxdt = \sup_{K \in S_2^*(T), \Vert K\Vert_{S_2^*(T)} = 1}\int\Lambda h|E^{(k)}|^2 dxdt,
    \end{aligned}
\end{equation}
where $S_2^*(T) = L_t^1H_x^{1/2}, h\colon [0, T] \times \mathbb{R}^2 \rightarrow \mathbb{C}$ is a solution to \eqref{101}.
Here we follow the same proof procedure when controlling the term $\Vert v^{(1)}\Vert_{S_2(T)}$, that is to say, applying div-curl lemma to the mass and momentum balance laws satisfied by $h$ and $E^{(k - 1)}$ to obtain the corresponding bilinear estimates, and hence we can control the term 
\begin{equation*}
    \int\Lambda h|E^{(k)}|^2 dxdt
\end{equation*}
in a procedure as in the proof of Proposition \ref{lemma l2}. 

Indeed, recall \eqref{105} and \eqref{123}. Note that by Proposition \ref{p5.2}, the terms 
\begin{equation}\label{132}
    \Vert E^{(k + 1)}\Vert_{S_1(T)}, \Vert v^{(k)}E^{(k)}\Vert_{N_1(T)}
\end{equation}
can be controlled by $CT^{\frac{1}{4}}$. It follows that we can still use div-curl lemma to derive the bilinear estimates for the Littlewood-Paley components of $h$ and $E^{(k)}$, and hence we have
\begin{equation}\label{68}
    \int\Lambda h|E^{(k)}|^2 dxdt \leq CT^{\frac{1}{4}}\Vert K\Vert_{S_2^*(T)}\Vert E^{(k)}\Vert_{S_1(T)}\left(\Vert E_0\Vert_{L_x^2} + \Vert v_0\Vert_{H_x^{-\frac{1}{2}}}\right).
\end{equation}

Lastly, from Proposition \ref{p5.2}, \eqref{131A} and \eqref{68}
\begin{equation*}
    \Vert v^{(k)}\Vert_{S_2(T)} \leq C, \qquad \forall k
\end{equation*}
provided that $T$ is sufficiently small.

\subsection{Lipschitz continuity}\label{Section continuity}
It is straightforward to verify that $(E^{(k)}-E^{(k-1)})$ and $(v^{(k)}- v^{(k-1)})$ satisfy
\begin{equation*}
    \begin{cases}
    (\ii\partial_{t}+\Delta)(E^{(k)}-E^{(k-1)})=(\Re v^{(k-1)})E^{(k-1)}-(\Re v^{(k-2)})E^{(k-2)},\\ 
     (\ii\partial_{t}-\Lambda)(v^{(k)}-v^{(k-1)})=\Lambda | E^{(k-1)}|^{2}-\Lambda | E^{(k-2)}|^{2},
    \end{cases}
    \qquad k \geq 2
\end{equation*}
and 
\begin{equation*}
    \begin{cases}
    (\ii\partial_{t}+\Delta)(E^{(1)}-E^{(0)})=(\Re v^{(0)})E^{(0)},\\ 
     (\ii\partial_{t}-\Lambda)(v^{(1)}-v^{(0)})=\Lambda | E^{(0)}|^{2}.
    \end{cases}
\end{equation*}
We define
\begin{equation*}
    R_{k}:=\left\Vert (\Re v^{(k)})E^{(k)}-(\Re v^{(k-1)})E^{(k-1)}\right\Vert_{N_{1}(T)}, \qquad k \geq 1.
\end{equation*}

\begin{proposition}\label{5.2}
    For all $k \geq 3$, 
    \begin{equation}\label{iteration of N}
        R_{k}  \lesssim  T^{\frac{1}{4}}R_{k - 1}+T^{\frac{1}{2}}R_{k-2}.
    \end{equation}
    \begin{proof}
        When $k \geq 3$, by triangle inequality
        \begin{equation}\label{128}
    \begin{aligned}
        &\left\Vert (\Re v^{(k)})E^{(k)}-(\Re v^{(k-1)})E^{(k-1)}\right\Vert_{N_{1}(T)}\\
        \leq &\left\Vert (\Re v^{(k)})(E^{(k)}-E^{(k-1)})\right\Vert_{N_{1}(T)}+\left\Vert (\Re v^{(k)}-\Re v^{(k-1)})E^{(k-1)}\right\Vert_{N_{1}(T)} \\ 
        := &R_{ka} + R_{kb}.
    \end{aligned} 
\end{equation}
    For $R_{ka}$, similar to the proof of Proposition \ref{p5.2}, using the bilinear estimates derived by applying div-curl lemma to the mass and momentum balance laws satisfied by frequency components of $v^{(k)}$ and $(E^{(k)} - E^{(k - 1)})$, and Strichartz estimates, we have 
    \begin{equation}\label{difference E}
\begin{split}
    R_{ka} &\leq \left\Vert  v^{(k)}(E^{(k)}-E^{(k-1)})\right\Vert_{N_{1}(T)}\\
   &\lesssim T^{\frac{1}{4}}  \left\Vert (E^{(k)}-E^{(k-1)})\left((\Re v^{(k-1)})E^{(k-1)}-(\Re v^{(k-2)})E^{(k-2)}\right)\right\Vert_{L_{t,x}^{1}}^{\frac{1}{2}}\\
    &\lesssim T^{\frac{1}{4}}\left\Vert E^{(k)}-E^{(k-1)}\right\Vert_{S_{1}(T)}^{\frac{1}{2}}  \left\Vert (\Re v^{(k-1)})E^{(k-1)}-(\Re v^{(k-2)})E^{(k-2)}\right\Vert_{N_{1}(T)}^{\frac{1}{2}}\\
    &\lesssim T^{\frac{1}{4}}\left\Vert (\Re v^{(k-1)})E^{(k-1)}-(\Re v^{(k-2)})E^{(k-2)}\right\Vert_{N_{1}(T)} \\ 
    &= T^{\frac{1}{4}}R_{k - 1}.
\end{split} 
\end{equation}
Repeating the same procedure above to $(v^{(k)} - v^{(k - 1)})$ and $E^{(k)}$ yields
\begin{equation}\label{130}
\begin{split}
    R_{kb} \lesssim \left\Vert (v^{(k)}- v^{(k-1)})E^{(k-1)}\right\Vert_{N_{1}(T)} \lesssim T^{\frac{1}{4}}\left\Vert(\Re v^{(k)}-\Re v^{(k-1)})(| E^{(k-1)}|^{2}-| E^{(k-2)}|^{2}) \right\Vert_{L_{t,x}^{1}}^{\frac{1}{2}},
\end{split}   
\end{equation}
and it follows from H\"{o}lder's inequality, div-curl lemma and the mass and momentum balance laws of frequency components $(\Re v^{(k)} - \Re v^{(k - 1)})$ and $(E^{(k)} - E^{(k - 1)})$ that 
\begin{equation}\label{131}
    \begin{split}
        &\left\Vert(\Re v^{(k)}-\Re v^{(k-1)})(| E^{(k-1)}|^{2}-| E^{(k-2)}|^{2}) \right\Vert_{L_{t,x}^{1}}\\
        \leq &\left\Vert(\Re v^{(k)}-\Re v^{(k-1)})E^{(k-1)}(  \overline{E^{(k-1)}}- \overline{E^{(k-2)}}) \right\Vert_{L_{t,x}^{1}}\\
        +&\left\Vert(\Re v^{(k)}-\Re v^{(k-1)})\overline{E^{(k-2)}}(  E^{(k-1)}- E^{(k-2)}) \right\Vert_{L_{t,x}^{1}}\\
        \leq &(\Vert E^{(k-1)}\Vert_{S_{1}(T)}+\Vert E^{(k-2)}\Vert_{S_{1}(T)}) \left\Vert(\Re v^{(k)}-\Re v^{(k-1)})(  E^{(k-1)}- E^{(k-2)}) \right\Vert_{N_{1}(T)}\\
        \lesssim &T^{\frac{1}{4}}\left\Vert (\Re v^{(k)}-\Re v^{(k-1)})(| E^{(k-1)}|^{2}-| E^{(k-2)}|^{2}) \right\Vert_{L_{t,x}^{1}}^{\frac{1}{2}}\left\Vert (\Re v^{(k-2)})E^{(k-2)}-(\Re v^{(k-3)})E^{(k-3)}\right\Vert_{N_{1}(T)}.
    \end{split}
\end{equation} 
Combining \eqref{130} with \eqref{131}, we have 
\begin{equation}\label{difference v}
    \begin{split}
       R_{kb}\lesssim T^{\frac{1}{2}} \left\Vert (\Re v^{(k-2)})E^{(k-2)}-(\Re v^{(k-3)})E^{(k-3)}\right\Vert_{N_{1}(T)} = T^{\frac{1}{2}}R_{k - 2}.
    \end{split}
\end{equation}
Now the desired conclusion follows from \eqref{128}, \eqref{difference E} and \eqref{difference v}.
    \end{proof}
\end{proposition}

From Proposition \ref{p5.2}-\ref{5.2}, we can choose $T\ll 1$ sufficiently small such that
\begin{equation*}
    R_{k}\leq \frac{1}{2^{k}},\qquad \forall k\geq 3.
\end{equation*}
This proves the contraction of the sequence $\{\Re v^{(k)} - \Re v^{(k - 1)}\}$.
Moreover, from Strichartz estimate \eqref{stricahrtz d=2} we have 
\begin{equation*}
    \left\Vert E^{(k+1)}-E^{(k)} \right\Vert_{X(T)}\lesssim R_{k}.
\end{equation*}
which completes the proof of \eqref{E_Hs}. The proof for \eqref{v_S2} follows similarly.

\section{Proof of Proposition \ref{prop} when \texorpdfstring{$d = 3$}{d=3} }\label{d=3}

The difference between the proofs of Proposition \ref{prop} when $d=3$ and $d=2$ originates from the Strichartz estimate \eqref{stricahrtz d=3}.
Here we only present the proof for controlling the term $\Vert E^{(1)}\Vert_{S_1(T)}$, as the remaining arguments parallel the case $d = 2$.

By Strichartz estimate \eqref{stricahrtz d=3},
\begin{equation*}
    \Vert E^{(1)}\Vert_{S_{1}(T)}\lesssim T^{\frac{s}{2}}\Vert E_{0}\Vert_{H^{s}_{x}}+T^{s}\Vert (\Re v^{(0)})E^{(0)}\Vert_{N_{1}(T)}\leq  T^{\frac{s}{2}}\Vert E_{0}\Vert_{H^{s}_{x}}+T^{s}\Vert  v^{(0)}E^{(0)}\Vert_{N_{1}(T)},
\end{equation*}
where $s > 0$. By duality,
\begin{equation*}
    \begin{aligned}
        \Vert  v^{(0)}E^{(0)}\Vert_{N_1(T)} = \sup_{\Vert G\Vert_{N^{*}_1(T)} = 1}\int \overline{G} v^{(0)}E^{(0)} dxdt.
    \end{aligned}
\end{equation*}
where 
\begin{equation*}
    \Vert u\Vert_{N^{*}_1(T)} := \sum_{i = 1}^N\left(\sum_{\lambda\geq 1}\lambda^{-4s}\Vert P_{\lambda}u\Vert_{L_{t}^{2}L_{\omega_i}^{2}L_{\omega_{i\bot}}^{\infty}}^{2}\right)^{\frac{1}{2}}.
\end{equation*}
For such $G$, we have 
    \begin{equation*}
            \begin{aligned}
                \int \overline{G} v^{(0)}E^{(0)} dxdt = \sum_{\sigma \geq 1}\int\overline{G}P_{\sigma}\left( v^{(0)}E^{(0)}\right) dxdt.
    \end{aligned}
\end{equation*}
It follows that
        \begin{equation*}
            \begin{aligned}
                \int \overline{G} v^{(0)}E^{(0)} dxdt &= \sum_{\sigma \geq 1}\int\overline{G}P_{\sigma}\left( v^{(0)}E^{(0)}\right) dxdt\\ &= \left(\sum_{1\leq\lambda\ll \mu\sim\sigma}+\sum_{1\leq\mu\ll \lambda\sim\sigma}+\sum_{1\leq\sigma\lesssim\lambda\sim \mu}\right)\int\overline{P_{\sigma}G}P_{\lambda} v^{(0)}P_{\mu}E^{(0)} dxdt \\ 
                &=: I_1 + I_2 + I_3. 
            \end{aligned}
\end{equation*} 

\begin{proposition}
    There holds 
    \begin{equation*}
        I_1, I_2, I_3 \lesssim \Vert E_0\Vert_{H_x^s}\Vert v_0\Vert_{H_x^{s-\frac{1}{2}}}.
    \end{equation*}
    \begin{proof}
        For simplicity we write $E :=E^{(0)}$ and $v :=v^{(0)}$.
        As we have shown in the proof of Proposition \ref{lemma l2}, for $I_1$ and $I_3$, by H\"{o}lder's inequality and bilinear estimate \eqref{High E} we have
\begin{align*}
\begin{split}
    I_{1} 
    &\leq \sum_{1\ll\mu\sim\sigma}\sum_{i = 1}^N\int \overline{  P_{\sigma}G}P_{\ll \mu} vP_{\mu, \omega_{i}}Edxdt\\
    &\leq \sum_{1\ll\mu\sim\sigma}\sum_{i = 1}^N\Vert  P_{\sigma}G\Vert_{L_{t}^{2}L_{\omega_{i}}^{2}L_{\omega_{i\bot}}^{\infty}}\Vert P_{\ll \mu} v P_{\mu, \omega_{i}}E\Vert_{L_{t}^{2}L_{\omega_{i}}^{2}L_{\omega_{i\bot}}^{1}} \\
    &\lesssim \sum_{1\ll\mu\sim\sigma}\sum_{i = 1}^N\mu^{-\frac{1}{2}}\Vert  P_{\sigma}G\Vert_{L_{t}^{2}L_{\omega_{i}}^{2}L_{\omega_{i\bot}}^{\infty}}\Vert P_{\ll \mu}v_0\Vert_{L_x^2} \Vert P_{\mu, \omega_{i}}E_0\Vert_{L_x^2} \\
    &\lesssim \Vert v_0\Vert_{H_x^{s-\frac{1}{2}}}\sum_{1\ll\mu\sim\sigma}\sum_{i = 1}^N\sigma^{-s}\Vert  P_{\sigma}G\Vert_{L_{t}^{2}L_{\omega_{i}}^{2}L_{\omega_{i\bot}}^{\infty}}\Vert P_{\mu, \omega_{i}}E_0\Vert_{L_x^2} \\
    &\leq \Vert v_0\Vert_{H_x^{s-\frac{1}{2}}}\left(\sum_{\sigma \geq 1}\sum_{i = 1}^N\sigma^{-4s}\Vert  P_{\sigma}G\Vert^2_{L_{t}^{2}L_{\omega_{i}}^{2}L_{\omega_{i\bot}}^{\infty}}\right)^{\frac{1}{2}}\left(\sum_{\sigma \geq 1}\sum_{i = 1}^N\sigma^{2s}\Vert P_{\sigma, \omega_{i}}E_0\Vert_{L_x^2}^2\right)^{\frac{1}{2}} \\
    &\lesssim \Vert v_0\Vert_{H_x^{s-\frac{1}{2}}}\Vert G\Vert_{N^{*}_1(T)}\Vert E_0\Vert_{H_x^s} \\
    &=\Vert v_0\Vert_{H_x^{s-\frac{1}{2}}}\Vert E_0\Vert_{H_x^s}
\end{split}  
\end{align*}
and 
\begin{equation*}
    \begin{aligned}
            I_3 &\lesssim \sum_{1 \lesssim \lambda \sim \mu}\sum_{i = 1}^N\int\overline{P_{\lesssim \mu}G}P_{\lambda} vP_{\mu, \omega_i}E dxdt \\ 
        &\lesssim \sum_{1 \lesssim \lambda \sim \mu}\sum_{i = 1}^N\Vert P_{\lesssim \mu}G\Vert_{L_t^2L_{\omega_i}^2L_{\omega_{i\bot}}^{\infty}} \Vert P_{\lambda} vP_{\mu, \omega_j}E\Vert_{L_t^{2}L_{\omega_i}^2L_{\omega_{i\bot}}^1} \\ 
        &\lesssim \Vert G\Vert_{N^{*}_1(T)}\sum_{1 \lesssim \lambda \sim \mu}\sum_{i = 1}^N\mu^{2s}\Vert P_{\lambda} vP_{\mu, \omega_i}E\Vert_{L_t^2L_{\omega_i}^2L_{\omega_{i\bot}}^1} \\ 
        &\lesssim \sum_{1 \lesssim \lambda \sim \mu}\sum_{i = 1}^N\mu^{2s-\frac{1}{2}}\Vert P_{\lambda}v_0\Vert_{L_x^2}\Vert P_{\mu, \omega_i}E_0\Vert_{L_x^2} \\ 
        &\lesssim \left(\sum_{\mu \geq 1}\mu^{2s-1}\Vert P_{\mu}v_0\Vert_{L_x^2}^2\right)^{\frac{1}{2}}\left(\sum_{\mu \geq 1}\sum_{i = 1}^N\mu^{2s}\Vert P_{\mu, \omega_i}E_0\Vert_{L_x^2}^2\right)^{\frac{1}{2}} \\ 
        &\lesssim \Vert v_0\Vert_{H_x^{s-\frac{1}{2}}}\Vert E_0\Vert_{H_x^s}.
    \end{aligned}
\end{equation*}
For $I_2$,  
\begin{equation*}
    \begin{split}
        I_{2} &\lesssim \sum_{1\ll \lambda\sim\sigma }\sum_{i = 1}^{N}\int \overline{ P_{\sigma}G}P_{\lambda,\omega_{i}} vP_{\ll\lambda}E dxdt\\
        &\leq \sum_{1\ll \lambda\sim\sigma}\sum_{i = 1}^N\Vert  P_{\sigma}G\Vert_{L_{t}^{2}L_{\omega_{i}}^{2}L_{\omega_{i\bot}}^{\infty}}\Vert  P_{\lambda,\omega_{i}} v P_{\ll \lambda}E\Vert_{L_{t}^{2}L_{\omega_{i}}^{2}L_{\omega_{i\bot}}^{1}}\\
        &\lesssim \sum_{\substack{1\ll \lambda\sim\sigma}}\sum_{i = 1}^N\lambda^{-\frac{1}{2}}\Vert  P_{\sigma}G\Vert_{L_{t}^{2}L_{\omega_{i}}^{2}L_{\omega_{i\bot}}^{\infty}}\Vert P_{\lambda,\omega_{i}} v_{0}\Vert_{L_{x}^{2}}\Vert P_{\ll\lambda}E_{0}\Vert_{L_{x}^{2}}\\
        &\lesssim \Vert E_{0}\Vert_{H_x^s}\left( \sum_{\substack{\lambda \geq 1}}\sum_{i = 1}^N\lambda^{-4s}\Vert  P_{\lambda}G\Vert_{L_{t}^{2}L_{\omega_{i}}^{2}L_{\omega_{i\bot}}^{\infty}}^{2}\right)^{\frac{1}{2}}\left( \sum_{\substack{\lambda\geq 1}}\sum_{i=1}^{N}\lambda^{2s -1}\Vert P_{\lambda,\omega_{i}} v_{0}\Vert_{L_{x}^{2}}^{2}\right)^{\frac{1}{2}}\\
        &\lesssim \Vert v_{0}\Vert_{H_x^{s - \frac{1}{2}}}\Vert E_{0}\Vert_{H_{x}^{s}}\Vert G\Vert_{S_{1}(T)}\\
        &=\Vert v_{0}\Vert_{H_x^{s - \frac{1}{2}}}\Vert E_{0}\Vert_{H_{x}^{s}}.
    \end{split}
\end{equation*}
    \end{proof}
\end{proposition}

In conclusion, we have
\begin{equation*}
    \Vert E^{(1)}\Vert_{S_{1}(T)}\lesssim T^{\frac{s}{2}}\Vert E_{0}\Vert_{H^{s}_{x}}+T^{s}\Vert  v^{(0)}E^{(0)}\Vert_{N_{1}(T)}\lesssim T^{\frac{s}{2}}\Vert E_{0}\Vert_{H^{s}_{x}}+T^{s}\Vert v_{0}\Vert_{H^{s-\frac{1}{2}}}\Vert E_{0}\Vert_{H_{x}^{s}}.
\end{equation*}
Note that the Strichartz estimates for the case $d = 3$ only introduces additional factors $T^{s/2}$ and $T^s$ compared to the case $d = 2$, and the impact of this difference is negligible when following the same proof procedure in Section \ref{d=2}. 
Therefore, we omit the rest of the proof.

\vspace{3mm}

\noindent\textbf{Acknowledgements.} This work was supported by the National Natural Science
Foundation of China [No. 12171097], the Key Laboratory of Mathematics for Nonlinear
Sciences (Fudan University), the Ministry of Education of China, Shanghai Key
Laboratory for Contemporary Applied Mathematics.

\vspace{3mm}

\noindent\textbf{Data availability.} The manuscript has no associated data.

\vspace{6mm}

\noindent\textbf{\Large Declaration}

\vspace{3mm}

\noindent\textbf{Conflict of interest.} The authors state that there is no conflict of interest.

\bibliographystyle{abbrv}
\small\bibliography{reference}
\end{document}